\documentclass[review]{elsarticle}

\usepackage{lineno,hyperref}
\modulolinenumbers[5]
\usepackage{amsfonts,amssymb}
\usepackage{amsmath}
\usepackage{amsthm}
\usepackage{mathrsfs} 
\usepackage{graphicx}
\usepackage{enumerate}
\usepackage{tikz}
\usepackage{caption}

\journal{Stochastic Processes and their Applications}

\newtheorem{thm}{Theorem}[section]
\newtheorem{prop}[thm]{Proposition}

\newtheorem{lem}[thm]{Lemma}
\theoremstyle{remark}
\newtheorem{rmk}[thm]{Remark}
\theoremstyle{definition}
\newtheorem{defn}{Definition}[section]

\definecolor{DarkGreen}{RGB}{0, 118, 0}
\definecolor{Enrica}{RGB}{200, 100, 20}
\DeclareMathOperator{\Cov}{Cov}
\DeclareMathOperator{\E}{\mathbb{E}}

\DeclareMathOperator{\N}{\mathbb{N}}
\DeclareMathOperator{\R}{\mathbb{R}}

\DeclareMathOperator{\cF}{\mathcal{F}}

\DeclareMathOperator{\cT}{\mathcal{T}}
\DeclareMathOperator{\cN}{\mathcal{N}}

\DeclareMathOperator{\cA}{\mathcal{A}}
\DeclareMathOperator{\cI}{\mathcal{I}}
\DeclareMathOperator{\cZ}{\mathcal{Z}}

\DeclareMathOperator{\cL}{\mathcal{L}}
\DeclareMathOperator{\cJ}{\mathcal{J}}

\DeclareMathOperator{\cS}{\mathcal{S}}
\DeclareMathOperator{\D}{\mathbb{D}}

\DeclareMathOperator{\bP}{\mathbb{P}}

\newcommand{\pd}[2]{\frac{\partial #1}{\partial #2}}
\newcommand{\pdsup}[3]{\frac{\partial^{#3} #1}{\partial #2^{#3}}}
\newcommand{\pdmis}[3]{\frac{\partial^{2} #1}{\partial #2 \ \partial #3}}

\newcommand{\der}[2]{\frac{d #1}{d #2}}

\newcommand{\Norm}[2]{\left\Vert #1 \right\Vert_{#2}}

\DeclareMathOperator*{\argmin}{argmin}

\begin{document}

\begin{frontmatter}

\title{An optimal Gauss-Markov approximation for a process with stochastic drift and applications}

\author[1]{Giacomo Ascione}
\ead{giacomo.ascione@unina.it}

\author[2]{Giuseppe D'Onofrio}
\ead{giuseppe.donofrio@unito.it}

\author[3]{Lubomir Kostal}
\ead{kostal@biomed.cas.cz}

\author[1]{Enrica Pirozzi \corref{cor1}}
\ead{enrica.pirozzi@unina.it}

\address[1]{Dipartimento di Matematica e Applicazioni ``Renato Caccioppoli'', Universit\`a degli Studi di Napoli Federico II, 80126 Napoli, Italy}
\address[2]{Dipartimento di Matematica “G. Peano”, Universit\`a degli Studi di Torino, Via Carlo Alberto 10, 10123 Torino, Italy}
\address[3]{Institute of Physiology of the Czech Academy of Sciences, Videnska 1083, 14220 Prague 4, Czech Republic}
\cortext[cor1]{Corresponding author}

\begin{abstract}
We consider a linear stochastic differential equation with stochastic drift. We study the problem of approximating the solution of such equation through an Ornstein-Uhlenbeck type process, by using direct methods of calculus of variations. We show that general power cost functionals satisfy the conditions for existence and uniqueness of the approximation. We provide some examples of general interest and we give bounds on the goodness of the corresponding approximations. Finally, we focus  on a model of a neuron embedded in a simple network and we study the approximation of its activity, by exploiting the aforementioned results.
\end{abstract}


\end{frontmatter}

\linenumbers
 \section{Introduction}\label{Sec.1}
For more than a century stochastic differential equations (SDEs) have played a key role in the description of fluctuating phenomena belonging to different areas of applied mathematics (\cite{karatzas},\cite{oksendal}). 
Here, we consider the following SDE in which the drift is characterized by a stochastic process $z(t)$ independent of $W(t)$: 
\begin{equation}\label{Equ.1.1}
dX(t)=[a(t)X(t)+z(t)]dt+\sigma dW(t), \qquad X(0)=x_0
\end{equation} 
where $a,z,\sigma,x_0$ are chosen to guarantee existence and uniqueness of the strong solution of the equation.
These equations are of interest in many applications.
In mathematical finance stochastic volatility  is used to model option pricing to represent that volatility varies  with respect to strike price and expiry (\cite{davis},\cite{lions}).
In time series analysis the stochastic trend is used in a difference equation. This is a discrete counterpart of the SDE \eqref{Equ.1.1} (\cite{gil}).
In computational neuroscience they model networks of interacting neurons in the presence of random synaptic weights (\cite{Delarue}, \cite{faugeras2019meanfield}, \cite{ostr}).
The present work also stems from neuronal modeling (see \cite{bress}, \cite{chaos}, \cite{Lansky2016}, \cite{Lansky2006},  \cite{Pirozzi2018}, \cite{TouboulSIAM}, \cite{tuckwellvol2}): one can model the  membrane potential of a neuron through a stochastic process $V(t)$ solving
\begin{equation*}\label{Equ.1.2}
dV(t)=\left(-\frac{V(t)-V_R}{\theta}+z(t)\right)dt+\sigma dW(t), \quad V(0)=v_0
\end{equation*}
where $V_R$ is the resting potential, $\theta$ is the characteristic time constant of the neuron and $z(t)$ is a process representing the collection of the stimuli the neuron under consideration receives from other neurons or from its own activity.\\ 
In the first part of this paper we study some features of the solution $X(t)$ of Equation \eqref{Equ.1.1}.
Since the process $X(t)$ depends on $z(t)$, it can be non Markov and/or non Gaussian.
This work is mainly focused on obtaining an optimal (in a sense that will be specified later) Gauss-Markov (GM) approximation of the process $X(t)$.
This approximation strategy enables one to use the extensive  theoretical results on  GM processes (see for instance \cite{mehr1965certain,MCAP2011, dinar, RdM2015}). Indeed, finding a good approximating GM process with a small approximation error allows one to use the integral equation approach to study the first passage time of the approximating process in place of the actual one.\\
To find the ``\textit{best}'' approximating GM process we look for the minimizer of a general cost functional $\mathcal J$, here usually a $L^2$ functional, among all Ornstein-Uhlenbeck type processes.

To do that, we solve the minimization problem in a very general setting.
Using results of calculus of variations, we show that the minimizer exists and it is unique inside the aforementioned class, requiring relatively mild conditions. Moreover, the Euler-Lagrange equation and the transversality condition of the approximation problem are obtained. The proofs are generalizations in a probabilistic setting of the main tools of the theory of Direct Methods of Calculus of Variations (see \cite{Dacorogna}), referring in particular to the relaxation of a problem. The Euler-Lagrange equation and the transversality condition are instead found by using an approach that is typical to the Classical Methods of Calculus of Variations (see \cite{dacorogna2014introduction}).
In particular we consider, in the familiy of all suitable cost functionals, power costs that are shown to satisfy the needed assumptions. Power costs represent the integral mean power error of the approximation. For instance, the second power cost is the integral of mean square error of approximation, whose minimizer gives a continuous-time version of a least mean square approximation.

Some examples are given in the cases in which $z(t)$ is a step function, a Poisson process, a compound Poisson process, a shot noise, a Brownian motion or an Ornstein-Uhlenbeck process.
Finally, we propose a stochastic neuronal model for the description of the firing activity of a neuron subject to the inputs coming from other neurons. Our example corresponds to the case in which $z(t)$ is a shot noise, that is to say $z(t)=\sum_{i=1}^{M}\beta_iR(t-\cT_i)$ where $\cT_i$ are i.i.d. random variables, distributed as a given variable $\cT$ with $\bP(\cT<0)=0$, independent of $W(t)$, $\beta_i$ are i.i.d. random variables independent of $(\cT_i)_{i \in \N}$ and $W(t)$ and $R$ is the response function, such that $R(t)=0 \quad \forall t<0$.
The stochastic nature of the drift in the model equation is due to the stochastic behavior of the inputs received from the other neurons that occur randomly in time and in space (see for instance \cite{abbott, tuckwellvol2, Coupled, carfora}).
In a theoretical context, we can adopt a specified distribution function for $\cT_i$ and $\beta_i$ whereas in the application context a distribution function may be one of the unknowns of the problem. This case is investigated in Section \ref{Sec.5}.
We stress that, although used here in the neuronal modeling context, the results obtained about the approximation are completely general for equations like \eqref{Equ.1.1}.
The paper is structured as follows:
\begin{itemize}
	\item In Section \ref{Sec.2} we introduce the target equation and the approximation problem;
	\item In Section \ref{Sec.3} we describe the problem in a more general setting and prove some sufficient and necessary conditions for existence and uniqueness of the approximation;
	\item In Section \ref{Sec.4} we provide some examples of general interest;
	\item In Section \ref{Sec.5} we construct a simple neuronal model and we use the previous results to find the approximating Gauss-Markov process for the membrane potential;
	\item Finally, in Section \ref{Sec.6} we summarize the work and we give some concluding remarks.
\end{itemize}  
%
%

\section{The linear equation and the $OU$ class}\label{Sec.2}
\subsection{The linear equation}\label{Sec.2.1}
Let $(\Omega, \mathcal{F},\mathbb{P})$ be a probability space endowed with the (completed) natural filtration $\{\mathcal{F}_t\}_{t\geq 0}$ of the standard Brownian motion $\{W(t),t\geq 0\}$. Let us consider the following stochastic differential equation in a time interval $[0,T]$, for a fixed $T<+\infty$,
\begin{equation}\label{Equ.2.1}
dX(t)=[a(t)X(t)+z(t)]dt+\sigma dW(t), \quad X(t)=x_0
\end{equation}
with $z(t)$ stochastic process adapted to $\{\mathcal{F}_t\}_{t\geq 0}$ such that its sample paths belong to $L^1(0,T)$, $a \in L^1(0,T)$ a damping rate, and $\sigma>0$. Let us denote by \linebreak $\cL^0(\Omega,\{\mathcal{F}_t\}_{t\geq 0};L^1(0,T))$ the space of the stochastic processes adapted to \linebreak $\{\mathcal{F}_t\}_{t\geq 0}$ with sample paths a.s. in $L^1(0,T)$ and by $\cL^p(\Omega,\{\mathcal{F}_t\}_{t\geq 0};L^1(0,T))$, for some $p \ge 1$, the space of the stochastic processes adapted to $\{\mathcal{F}_t\}_{t\geq 0}$ with sample paths a.s. in $L^1(0,T)$ such that $\E[|z(t)|^p]<+\infty$. For simplicity we will assume $z(t)$ to be independent from $W(t)$. In the examples the function $a$ will be a negative constant.\\
Since Eq.\eqref{Equ.2.1} is a linear equation, one can ensure the existence of a unique strong solution (see for instance \cite{karatzas}). In particular one has the following result:
\begin{prop}\label{prop2.1}
	The map $\cS:\cL^0(\Omega,\{\cF_t\}_{t\ge0};L^1(0,T)) \to \cL^0(\Omega;L^0(0,T))$ given by
	\begin{equation*}
	\cS z(t)=e^{A(t)}x_0+\sigma e^{A(t)}\int_0^{t}e^{-A(s)}dW_s+e^{A(t)}\int_0^tz(s)e^{-A(s)}ds,
	\end{equation*}
	where 
	\begin{equation*}
	A(t)=\int_0^ta(s)ds,
	\end{equation*}
	is an injection that associates $z \in \cL^0(\Omega,\{\mathcal{F}_t\}_{t\geq 0};L^1(0,T))$ to the unique strong solution of Equation \eqref{Equ.2.1}.
\end{prop}
The well-posedness of $\cS$, i.e. existence and uniqueness of the strong solution, follows from \cite[Theorem $2.5$ and $2.13$]{karatzas}. Later we will prove that $\cS$ is an injection.\\
Given a generic $z(t) \in \cL^0(\Omega,\{\cF_t\}_{t\ge0}; L^1(0,T))$, we can split the process $X(t)=\cS z(t)$ in two parts. Indeed, if we set
\begin{equation}\label{Equ.2.3}
Y(t)=e^{A(t)}x_0+\sigma e^{A(t)}\int_0^{t}e^{-A(s)}dW_s, \quad Z(t)=e^{A(t)}\int_0^tz(s)e^{-A(s)}ds 
\end{equation}
we have
\begin{equation}\label{spliteq}
X(t)=Y(t)+Z(t).
\end{equation}
In particular $Y(t)$ is an Ornstein-Uhlenbeck process independent of $Z(t)$. Its mean and covariance are given by 
\begin{equation}\label{Equ.2.4}
\mathbb{E}[Y(t)]=e^{A(t)}x_0, \quad \mathrm{Cov}(Y(t),Y(s))=\sigma^2e^{A(t)+A(s)}\int_0^{\min\{t,s\}}e^{-2A(u)}du.
\end{equation}
On the other hand, if $z(t)$ is a Riemann-integrable Gaussian process (for instance if it admits continuous sample paths), then $Z(t)$ is also Gaussian and independent of $Y(t)$. By Equation \eqref{spliteq} we conclude that, in such case, $X(t)$ is a Gaussian process. In general $X(t)$ could be neither Markov nor Gaussian.\\
Let us state this easy Lemma.
\begin{lem}
	Let $z \in \cL^1(\Omega,\{\mathcal{F}_t\}_{t\geq 0};L^1(0,T))$ such that $\E[z(t)] \in L^1(0,T)$. Then
	\begin{equation}\label{meanZ}
	\E[Z(t)]=e^{A(t)}\int_0^t \E[z(s)]e^{-A(s)}ds, \ \forall t \in [0,T].
	\end{equation}
	Moreover, if $z \in \cL^2(\Omega,\{\mathcal{F}_t\}_{t\geq 0};L^1(0,T))$ is such that $\D[z(t)]=\E[(z(t)-\E[z(t)])^2] \in L^1(0,T)$, then
	\begin{equation}\label{covZ}
	\Cov(Z(t),Z(s))=e^{A(t)+A(s)}\int_0^t\int_0^s\Cov(z(u),z(v))e^{-A(u)-A(v)}dudv.
	\end{equation}
\end{lem}
The proof follows from the application of Fubini's theorem.\\
The previous Lemma, together with the independence of $Z(t)$ and $Y(t)$ and Equation \eqref{spliteq}, gives us the following Proposition.
\begin{prop}\label{prop_mean_cov}
	If $z \in \cL^2(\Omega,\{\mathcal{F}_t\}_{t\geq 0};L^1(0,T))$ is such that $\D[z(t)] \in L^1(0,T)$ then
	\begin{equation*}
	\E[X(t)]=e^{A(t)}\left(x_0+\int_0^t \E[z(s)]e^{-A(s)}ds\right).
	\end{equation*}
	and, for $s\le t$
	\begin{multline*}
	\Cov(X(t),X(s))=e^{A(t)+A(s)}\times\\\times\left[\sigma^2\int_0^{s}e^{-2A(u)}du+\int_0^t\int_0^s\Cov(z(u),z(v))e^{-A(u)-A(v)}dudv\right].
	\end{multline*}
\end{prop}
\subsection{The $OU(a,\sigma,x_0)$ class}\label{Sec.2.2}
A particular solution of \eqref{Equ.2.1} is achieved when $z(t)$ is a degenerate stochastic process (i.e. a deterministic function). Indeed, for a function $f \in L^1(0,T)$ let us consider the equation
\begin{equation}\label{SDEapprox}
dX^f(t)=[a(t)X^f(t)+f(t)]dt+\sigma dW(t), \quad X(0)=x_0.
\end{equation}
The solution map $\cS$ can be still used, since we can consider \linebreak $L^1(0,T)\subset \cL^0(\Omega,\{\cF_t\}_{t \ge 0}; L^1(0,T))$ by identifying any deterministic function $f \in L^1(0,T)$ with the constant stochastic process $f(\omega)=f$ for any $\omega \in \Omega$. Thus we have that
\begin{equation*}
X^f(t)=\cS f(t)=e^{A(t)}x_0+\sigma e^{A(t)}\int_0^{t}e^{-A(s)}dW_s+e^{A(t)}\int_0^tf(s)e^{-A(s)}ds.
\end{equation*}
By using Equation \eqref{spliteq}, we have that $Z(t)$ is a deterministic function, hence it does not play any role in the auto-covariance function of $X^f$, which is now determined by \eqref{Equ.2.4}. $X^f$ is Gaussian since it is a sum of independent Gaussian processes and Markov property is ensured by the fact that is solution of \eqref{SDEapprox}. Now we define the class of processes of the form $X^f=\cS f$ for some $f \in L^1(0,T)$.
\begin{defn}
	The Ornstein-Uhlenbeck class $OU(a,\sigma,x_0)$ is defined as
	\begin{equation*}
	OU(a,\sigma,x_0):=\left\{X \in\mathcal{L}^2(\Omega,\{\cF_t\}_{t \ge 0},L^1(0,T)): \ \exists f \in L^1(0,T), \ X=\cS f \right\}.
	\end{equation*}
	We will denote $X^f \in OU(a,\sigma,x_0)$ to state that $X^f=\cS f$.
\end{defn}
As already mentioned in the introduction, a Gauss-Markov process is easier to handle than processes of the form $\cS z(t)$ for general $z \in \cL^0(\Omega,\{\cF_t\}_{t \ge 0}; L^1(0,T))$, and many existing tools and results about  these processes can be exploited. Thus it is interesting to understand how can we \textit{best} approximate a general solution of \eqref{Equ.2.1} with a process $X^f \in OU(a,\sigma,x_0)$.
\subsection{The approximation problem}\label{Sec.2.3}
Let us consider a process $z(t)\in \cL^0(\Omega,\{\mathcal{F}_t\}_{t\geq 0};L^1(0,T))$ and let us introduce a cost functional $\cJ$ on the class $OU(a,\sigma,x_0)$, defined, for any $X^f \in OU(a,\sigma,x_0)$, as
\begin{equation}\label{JXf}
\cJ[X^f]=\E\left[\int_0^T J(t,|X(t)-X^f(t)|)dt+\Phi(|X(T)-X^f(T)|)\right],
\end{equation}
for some functions $J(t,x)$ and $\Phi(x)$, where $X=\cS z$. The cost functional $\cJ$ represents the mean cost we are \textit{going to pay} for approximating the process $X$ with a process $X^f \in OU(a,\sigma,x_0)$. The function $J$ will be used to represent the running cost of the approximation, while $\Phi$ is the final cost. To find the \textit{best} approximation means that we want to find a process $\widetilde{X} \in OU(a,\sigma,x_0)$ such that
\begin{equation*}
\mathcal{J}[\widetilde{X}]=\min_{X^f \in OU(a,\sigma,x_0)}\cJ[X^f].
\end{equation*}
By using the definition of the solution map $\cS$ in Proposition \ref{prop2.1} and Equation \eqref{spliteq}, one obtains
\begin{equation*}
|X(t)-X^f(t)|=|Z(t)-F(t)|
\end{equation*}
where
\begin{equation*}
F(t)=e^{A(t)}\int_0^t f(s)e^{-A(s)}ds.
\end{equation*}
This means that actually
\begin{equation*}
\cJ[X^f]=\E\left[\int_0^T J(t,|Z(t)-F(t)|)dt+\Phi(|Z(T)-F(T)|)\right].
\end{equation*}
Let us now consider the space of absolutely continuous functions on $[0,T]$ (see \cite[Sections $6.4$ and $6.5$]{royden2010real}), i.e.
\begin{equation*}
AC([0,T]):=\left\{F \in C^0([0,T]): \ \exists f \in L^1(0,T), F(t)=F(0)+\int_0^tf(s)ds \ \forall t \in [0,T]\right\}
\end{equation*}
and let us define the map $\cI:L^1(0,T) \to \cA=\{F \in AC([0,T]): \ F(0)=0\}$ such that
\begin{equation*}
\cI f(t)=e^{A(t)}\int_0^t f(s)e^{-A(s)}ds.
\end{equation*}
This map is a bijection between $L^1(0,T)$ and $\cA$ since
it associates $f \in L^1(0,T)$ to the unique Caratheodory solution (see \cite[Theorem $1.1$ and $2.1$ of Section $2$]{cara}) $F$ of the Cauchy problem
\begin{equation*}
\begin{cases}
F'(t)=a(t)F(t)+f(t) & t \in (0,T),\\
F(0)=0.
\end{cases}
\end{equation*}
On the other hand for any $F \in \cA$ we have 
	\begin{equation}
	\cI^{-1}F(t)=F'(t)-a(t)F(t).
	\end{equation}
For any process $z \in \cL^0(\Omega,\{\cF_t\}_{t\ge0};L^1(0,T))$ we have, by Equation \eqref{spliteq}, $\cS z=Y+\cI z$. The fact that $\cI$ is a bijection proves the injectivity of $\cS$. Moreover, $\cS: \cL^0(\Omega,\{\cF_t\}_{t\ge0};L^1(0,T)) \to \cS(\cL^0(\Omega,\{\cF_t\}_{t\ge0};L^1(0,T)))$  is bijective and $OU(a,\sigma,x_0)\subset \cS(\cL^0(\Omega,\{\cF_t\}_{t\ge0};L^1(0,T)))$.\\
Now we can define a new functional directly on $\cA$ (that we will still denote with $\cJ$) that is the composition of the functional $\cJ$, the map $\cI$ and the inverse solution map $\cS^{-1}$ on $OU(a,\sigma,x_0)$, and is given by
\begin{equation*}
\cJ[F]=\E\left[\int_0^T J(t,|Z(t)-F(t)|)dt+\Phi(|Z(T)-F(T)|)\right],
\end{equation*}
for $F \in \cA$.
Being the maps $\cS$ and $\cI$ bijections, finding the minimizer $\widetilde{F}$ of $\cJ$ in $\cA$ gives us the best approximating process $\widetilde{X}=\cS \cI^{-1}\widetilde{F} \in OU(a,\sigma,x_0)$.\\
We can study these kind of cost functionals as particular cases of the more general cost functional
\begin{equation*}
\cJ[F]=\E\left[\int_0^T J(t,Z(t),F(t))dt+\Phi(Z(T),F(T))\right].
\end{equation*}
In particular, we want to find a $\widetilde{F} \in \cA$ such that
\begin{equation*}
\cJ[\widetilde{F}]=\min_{F \in \cA}\cJ[F].
\end{equation*}
In the following section we show that under some hypotheses this problem admits a unique solution and we find some necessary conditions that will be the main tools to actually find the minimizer $\widetilde{F} \in \cA$.
\section{Optimality conditions and existence of the solution of the approximation problem}\label{Sec.3}
\subsection{The main result}\label{Sec.3.1}
Let us state the problem in its full generality. Let us consider the stochastic process $Z(t) \in \cL^0(\Omega,\{\mathcal{F}_t\}_{t\geq 0};L^1(0,T))$ with a.s. continuous paths and let us define the probability measure flow $\mu_t=\cL[Z(t)]$ where  $\cL[X]$  denotes the law of a random variable $X$. 
Fix $T \ge 0$ and define 
\begin{equation*}
R_T=[0,T]\times \R.
\end{equation*}
For any measurable set $A \subseteq R_T$ define the section $A_t=\{x \in \R: \ (t,x)\in A\}$ for fixed $t \in [0,T]$. Then let us define the set function $\mu$ as follows
\begin{equation}\label{mu}
\mu(A)=\int_{0}^{T}\mu_t(A_t)dt \qquad \forall A\subseteq R_T \mbox{ measurable }.
\end{equation}
It is not difficult to check that $\mu$ is a measure.\\
One can also show that for any measurable function $f:R_T \to \R$ we have
\begin{equation}\label{intmu}
\int_{\Omega} f(t,z)\mu(dzdt)=\int_0^{T} \int_{\R}f(t,z)\mu_t(dz)dt=\int_0^{T}\E[f(t,Z(t))]dt.
\end{equation}
Consider now the functions 
\begin{equation}\label{JPhi}
J:(t,z,x)\in R_T \times \R \mapsto J(t,z,x)\in \R \ \mbox{ and } \  \Phi:(z,x)\in \R\times \R\mapsto \Phi(z,x)\in \R
\end{equation}
and define the functional $\mathcal{J}:AC([0,T])\to \R$ as
\begin{equation}\label{FunctionalJ}
\mathcal{J}[F]=\E\left[\int_0^{T}J(t,Z(t),F(t))dt+\Phi(Z(T),F(T))\right].
\end{equation}
We want to solve the following problem: 
\begin{equation}\label{prob}
\mbox{ find }\arg\min_{F \in \mathcal{A}}\mathcal{J}[F]
\end{equation}
where the admissible set is defined as
\begin{equation*}
\mathcal{A}=\{F \in AC([0,T]): \ F(0)=0\}.
\end{equation*}
We will consider the following assumptions, that we will explain while proving the main result:
\begin{itemize}
	\item[A1] There exists a function $F \in \cA$ such that $\cJ[F]<+\infty$;
	\item[A2] The functions $J(t,z,x)$ and $\Phi(z,x)$ defined in \eqref{JPhi} are non-negative for any $(t,z,x)\in R_T \times \R$;
	\item[A3] For fixed $(t,z) \in R_T$ the map $x \mapsto J(t,z,x)$ is in $C^1$;
	\item[A4] For any compact set $K \subset \R$ there exists a function $\Psi_K(t,z)\in L^1(R_T;\mu)$ such that
	\begin{equation*}
	\left|\pd{J}{x}(t,z,x)\right|\le \Psi_K(t,z), \quad \forall x \in K;
	\end{equation*}
	\item[A5] For fixed $z \in \R$ the map $x \mapsto \Phi(z,x)$ is in $C^1$;
	\item[A6] For any compact set $K \subset \R$ there exists a function $\Theta_K(z)\in L^1(\R;\mu_T)$ such that
	\begin{equation*}
	\left|\pd{\Phi}{x}(z,x)\right|\le \Theta_K(z) \quad \forall x \in K;
	\end{equation*}
	\item[A7] For any fixed $t \in [0,T]$, the map $x \mapsto \E[J(t,Z(t),x)]$ is strictly convex, decreasing as $x \to -\infty$ and increasing as $x \to +\infty$;
	\item[A8] There exist two constants $\alpha,M>0$, a function $h \in L^1(0,T)$ and an exponent $p>1$ such that for any $t \in [0,T]$ and $x \in \R$ with $|x|>M$
	\begin{equation*}
	\E[J(t,Z(t),x)]\ge \alpha(h(t)+|x|^p);
	\end{equation*}
	\item[A9] The map $x \mapsto \E[\Phi(Z(T),x)]$ is proper or constant;
	\item[A10] The map $x \mapsto \E[\Phi(Z(T),x)]$ is convex;
	\item[A11] The function
	\begin{equation*}
	(t,x)\mapsto \E[J(t,Z(t),x)]
	\end{equation*}
	is in $C^2(R_T \setminus \cN)$, where $\cN \subset R_T$ is such that $\cZ:=\{t \in [0,T]: \ \exists x \in \R, \ (t,x)\in \cN\}$ is at most finite, and $$\pdsup{}{x}{2}\E\left[J(t,Z(t),x)\right]>0$$ for any $(t,x)\not \in \cN$;
	\item[A12] The function
	\begin{equation}\label{eta}
	\eta(t)=-\frac{\pdmis{}{x}{t}\E[J(t,Z(t),F(t))]}{\pdsup{}{x}{2}\E[J(t,Z(t),F(t))]}
	\end{equation}
	belongs to $L^1(0,T)$, where $F(t)$ is a solution of
	\begin{equation*}
	\E\left[\pd{J}{x}(t,Z(t),F(t))\right]=0 \quad \forall t \in (0,T),
	\end{equation*}
	defined for $t \in \mathfrak{I}\supset [0,T]\setminus \cZ$, where $\cZ$ is at most a finite set;
	\item[A13] It holds true that
	\begin{equation*}
	\E\left[\pd{J}{x}(0,Z(0),0)\right]=0;
	\end{equation*}
	\item[A14] Given $a^*$ the (unique) solution of
	\begin{equation*}
	\E\left[\pd{J}{x}(T,Z(T),x)\right]=0
	\end{equation*}
	then it also holds
	\begin{equation*}
	\E\left[\pd{\Phi}{x}(Z(T),a^*)\right]=0.
	\end{equation*}
\end{itemize}
Although numerous, these assumptions are not so strict, neither unusual, as we will see later. Indeed, we will show that an important family of cost functions (the power costs) satisfies all the assumptions, depending on the regularity of the process $Z(t)$.Now we state the main result of the paper.
\begin{thm}\label{thmmain}
	Under assumptions $A1-A14$, there exists a unique solution $F^* \in \cA$ of \eqref{prob} and it is the unique solution of Equations
	\begin{equation}\label{eulan}
	\E\left[\pd{J}{x}(t,Z(t),F^*(t))\right]=0 \quad \forall t \in (0,T)
	\end{equation}
	and
	\begin{equation}\label{transv}
	\E\left[\pd{\Phi}{x}(Z(T),F^*(T))\right]=0.
	\end{equation}
\end{thm}
From a probabilistic point of view by Equations \eqref{eulan} and \eqref{transv} we are asking for $F^*(t)$ to be, on average, a critical point of the running cost $J(t,Z(t),x)$ and the final cost $\Phi(Z(T),x)$. Furthermore, under our assumptions, one can show that Equation \eqref{eulan} can be also written as
\begin{equation*}
\pd{}{x}\E[J(t,Z(t),F^*(t))]=0,
\end{equation*}
that is to say that $F^*(t)$ is also a critical point of the mean of the running cost. The same holds for the final cost.\\
The proof of Theorem \ref{thmmain} will be given in Subsections \ref{Sec.3.2}, \ref{Sec.3.3}, and \ref{Sec.3.4}. The proof is structured as follows:
\begin{itemize}
\item First, in Subsection \ref{Sec.3.2} we find necessary optimality conditions in terms of Equations \eqref{eulan} and \eqref{transv}, which are the Euler-Lagrange equation and the Transversality Condition of the function $\mathcal{J}$ (see for instance \cite[Chapter $2$]{dacorogna2014introduction}), by using Assumptions $A1-A6$;
\item In Subsection \ref{Sec.3.3}, to obtain an existence result, we  need to \textit{relax} the problem, in the spirit of Calculus of Variations (see for instance \cite[Section $1.4$]{Dacorogna}), by introducing a more general functional on a more general admissible set. For the relaxed functional we are able to prove lower semicontinuity and then existence of the minimizer, by using Assumptions $A1-A9$;
\item In Subsection \ref{Sec.3.4} we show that exactly one of the minimizers belong to the admissible set $\cA$ and thus it is a minimizer of the original problem, by using the whole set of Assumptions and completing the proof of the Theorem.
\end{itemize}
\subsection{Necessary optimality conditions}\label{Sec.3.2}
In this section we perform the first step of our plan. Indeed, by using Assumptions $A1-A6$ we will prove that any minimizer of $\mathcal{J}$ in $\cA$ solves Equations \eqref{eulan} and \eqref{transv}. These equations will be the main tools to actually find a minimizer for $\mathcal{J}$.
We have the following optimality conditions, by means of the Euler-Lagrange equation and the transversality condition. The proof follows the ideas of \cite[Theorem $2.1$ Part $1$]{dacorogna2014introduction}, adapted to our case.
\begin{thm}\label{thmnec}
	Under Assumptions $A1-A6$, let $F^* \in \mathcal{A}$ be a solution of the problem \eqref{prob}. Then $F^*$ is solution of Equations \eqref{eulan} and \eqref{transv}.
\end{thm}
\begin{proof}
	Let us first observe that if $F^* \in \mathcal{A}$ is a minimizer for $\cJ$, then, by Assumption $A1$, $\mathcal{J}[F^*]<+\infty$. In particular
	\begin{equation*}
	\E\left[\int_0^T J(t,Z(t),F^*(t))dt\right]<+\infty
	\end{equation*}
	and thus $\int_0^T J(t,Z(t),F^*(t))dt<+\infty$ almost surely. In particular we can use Fubini's theorem to obtain
	\begin{align*}
	\mathcal{J}[F^*]&=\int_0^{T}\E[J(t,Z(t),F^*(t))]dt+\E[\Phi(Z(T),F^*(T))]\\
	&=\int_{0}^{T}\int_{\R}J(t,z,F^*(t))\mu_t(dz)dt+\int_{\R}\Phi(z,F^*(T))\mu_T(dz)\\
	&=\int_{\Omega}J(t,z,F^*(t))\mu(dzdt)+\int_{\R}\Phi(z,F^*(T))\mu_T(dz).
	\end{align*}
	First of all, let us fix $\varphi \in C^\infty_c((0,T))$ and define $F^*_\varepsilon(t)=F^*(t)+\varepsilon \varphi(t)$. Since $\varphi(T)=0$, then we have that
	\begin{equation*}
	\mathcal{J}[F^*_\varepsilon]=\E\left[\int_0^T J(t,Z(t),F^*(t)+\varepsilon \varphi(t))dt+\Phi(Z(T),F^*(T))\right].
	\end{equation*}
	Now observe that since $\varphi \in C^\infty_c((0,T))$, then $F^*_\varepsilon \to F^*$ uniformly as $\varepsilon \to 0$. In particular let us consider a tubular neighbourhood of $F^*$, i.e. $$Q_\delta=\{(t,x) \in [0,T]\times \R: \ x \in [F^*(t)-\delta,F^*(t)+\delta]\}$$
	and a compact set $K \subset \R$ such that $Q_\delta \subseteq [0,T]\times K$. Then there exists a $\varepsilon_0$ such that for $\varepsilon \in (-\varepsilon_0,\varepsilon_0)$ the couples $(t,F^*(t)+\varepsilon \varphi(t))\in Q_\delta$ and then $F^*(t)+\varepsilon \varphi(t)\in K$. Hence we have, by Assumption $A4$,
	\begin{align*}
	|J(t,Z(t),F^*(t)+\varepsilon\varphi(t))|&\le|J(t,Z(t),F^*(t)+\varepsilon\varphi(t))-J(t,Z(t),F^*(t))|\\
	&+|J(t,Z(t),F^*(t))|\\
	&\le \Psi_K(t,Z(t)) |\varphi(t)|+|J(t,Z(t),F^*(t))|.
	\end{align*}
	Taking the mean and then integrating with respect to time in the right hand side we have, by Equation \eqref{intmu},
	\begin{align*}
	\int_{0}^{T}& \E[|\Psi_K(t,Z(t))| |\varphi(t)|+|J(t,Z(t),F^*(t))|]dt \\
	&\le \Norm{\varphi}{L^\infty(0,T)}\int_{\Omega}|\Psi_K(t,z)|\mu(dzdt)+\int_{\Omega}|J(t,Z(t),F^*(t)|\mu(dzdt)<+\infty. 
	\end{align*}
	In particular we can use Fubini's theorem to obtain
	\begin{equation*}
	\E\left[\int_0^T|\Psi_K(t,Z(t))| |\varphi(t)|+|J(t,Z(t),F^*(t))|dt\right]<+\infty
	\end{equation*}
	and then $\Psi_K(t,Z(t))\varphi(t)+J(t,Z(t),F^*(t))$ is almost surely in $L^1(0,T)$. This implies that $J(t,Z(t),F^*(t)+\varepsilon\varphi(t))$ is almost surely in $L^1(0,T)$ and in particular
	\begin{equation*}
	\mathcal{J}[F^*_\varepsilon]=\int_{\Omega}J(t,z,F^*(t)+\varepsilon\varphi(t))\mu(dzdt)+\int_{\R}\Phi(z,F^*(T))\mu_T(dz).
	\end{equation*}
	Consider the function $g:\varepsilon \in (-\varepsilon_0,\varepsilon_0)\mapsto \mathcal{J}[F^*_\varepsilon]$ and observe that it admits a minimum in $\varepsilon=0$. Let us show that $g$ is in $C^1$. To do this, let us consider the function
	\begin{equation*}
	h:(t,z,\varepsilon)\in R_T \times (-\varepsilon_0,\varepsilon_0)\mapsto J(t,z,F^*(t)+\varepsilon \varphi(t))
	\end{equation*}
	such that
	\begin{equation*}
	g(\varepsilon)=\int_{R_T}h(t,z,\varepsilon)\mu(dzdt)+\int_{\R}\Phi(z,F^*(T))\mu_T(dz)
	\end{equation*}
	and observe that, $J(t,z,x)$ being a $C^1$ function in $x$ by Assumption $A3$, we have
	\begin{equation*}
	\pd{h}{\varepsilon}(t,z,\varepsilon)=\pd{J}{x}(t,z,F^*(t)+\varepsilon \varphi(t))\varphi(t)
	\end{equation*}
	and in particular
	\begin{equation*}
	\left|\pd{h}{\varepsilon}(t,z,\varepsilon)\right|\le \Norm{\varphi}{L^\infty}\Psi_K(t,z)
	\end{equation*}
	thus we have a uniform (with respect to $\varepsilon$) $L^1$ bound on the derivative of $h$. By differentiation under the integral sign we have that $g \in C^1$ and
	\begin{align*}
	g'(\varepsilon)&=\int_{\Omega}\pd{J}{x}(t,z,F^*(t)+\varepsilon \varphi(t))\varphi(t)\mu(dzdt)\\&=\int_0^{T}\E\left[\pd{J}{x}(t,Z(t),F^*(t)+\varepsilon \varphi(t))\right]\varphi(t)dt.
	\end{align*}
	Now, by Fermat's theorem, we know that $g'(0)=0$ and then, by the fact that we arbitrarily chose $\varphi \in C^\infty_c(0,T)$,
	\begin{equation}\label{intEphi}
	\int_0^{T}\E\left[\pd{J}{x}(t,Z(t),F^*(t))\right]\varphi(t)dt=0 \ \quad \forall \varphi \in C^\infty_c((0,T)).
	\end{equation}
	By Fundamental Lemma of Calculus of variation (see \cite[Theorem $3.40$]{Dacorogna}), Equation \eqref{intEphi} implies
	\begin{equation*}
	\E\left[\pd{J}{x}(t,Z(t),F^*(t))\right]=0 \quad \forall t \in (0,T).
	\end{equation*}
	Now let us choose again $\varphi \in C^\infty_c((0,T])$. Working as before on $\cJ[F^*_\varepsilon]$, by using also Assumption $A6$ (since $\varphi(T)\not = 0$) we have
	\begin{equation*}
	\mathcal{J}[F^*_\varepsilon]=\int_{\Omega}J(t,z,F^*(t)+\varepsilon\varphi(t))\mu(dzdt)+\int_{\R}\Phi(z,F^*(T)+\varepsilon \varphi(T))\mu_T(dz).
	\end{equation*}
	As done before, let us introduce a function $g(\varepsilon)=\mathcal{J}[F^*_\varepsilon]$ and let us define the function
	\begin{equation*}
	k(z,\varepsilon)=\Phi(z,F^*_\varepsilon(T))
	\end{equation*}
	to obtain
	\begin{equation*}
	g(\varepsilon)=\int_{\Omega}h(t,z,\varepsilon)\mu(dzdt)+\int_{\R}k(z,\varepsilon)\mu_T(dz).
	\end{equation*}
	Now let us show that $g$ is in $C^1(-\varepsilon_0,\varepsilon_0)$. To do that, we only need to work with $k$. We have, by Assumption $A5$,
	\begin{equation*}
	\pd{k}{\varepsilon}(z,\varepsilon)=\pd{\Phi}{x}(z,F^*(T)+\varepsilon\varphi(T))\varphi(T)
	\end{equation*}
	and then
	\begin{equation*}
	\left|\pd{k}{\varepsilon}(z,\varepsilon)\right|\le \varphi(T) \Theta_K(z).
	\end{equation*}
	Thus we can differentiate under the integral sign, obtaining
	\begin{multline*}
	g'(\varepsilon)=\int_0^{T}\E\left[\pd{J}{x}(t,Z(t),F^*(t)+\varepsilon \varphi(t))\right]\varphi(t)dt+\\+\E\left[\pd{\Phi}{x}(Z(T),F^*(T)+\varepsilon \varphi(T))\right]\varphi(T).
	\end{multline*}
	By using Fermat's Theorem, we have $g'(0)=0$ and then
	\begin{equation*}
	\int_0^{T}\E\left[\pd{J}{x}(t,Z(t),F^*(t))\right]\varphi(t)dt+\E\left[\pd{\Phi}{x}(Z(T),F^*(T))\right]\varphi(T)=0.
	\end{equation*}
	However, we already proved that $\E\left[\pd{J}{x}(t,Z(t),F^*(t))\right]=0$, hence, since we have arbitrarily chosen $\varphi \in C^\infty_c((0,T])$
	\begin{equation*}
	\E\left[\pd{\Phi}{x}(Z(T),F^*(T))\right]\varphi(T)=0, \quad \forall \varphi \in C_c^\infty((0,T])
	\end{equation*}
	from which we finally obtain
	\begin{equation*}
	\E\left[\pd{\Phi}{x}(Z(T),F^*(T))\right]=0.
	\end{equation*}
\end{proof}
\begin{rmk}
	Let us observe that Assumption $A1$ is a non-triviality assumption, to avoid functionals of the form $\cJ \equiv +\infty$. Assumption $A2$ is used instead to avoid the case $\inf \cJ=-\infty$. Concerning Assumptions $A3-A6$, they are typical Assumptions of $C^1$ regularity and integrability of the local Lipschitz constant.
\end{rmk}
\subsection{Existence of a minimizer for a relaxed problem}\label{Sec.3.3}
Now we introduce a relaxed problem. Indeed, we are not able to prove directly existence of the minimizer in the admissible set $\cA$. Hence we will ``enlarge'' this set and extend the functional in order to prove an existence result. This relaxation technique is typical of direct methods of Calculus of Variations (see \cite{Dacorogna}).
From now on we will split the functional $\mathcal{J}$ in two parts
\begin{align*}
\mathcal{I}_1[F]&=\E\left[\int_0^T J(t,Z(t),F(t))dt\right]\\
\mathcal{I}_2[F]&=\E[\Phi(Z(T),F(T))]
\end{align*}
such that $\mathcal{J}[F]=\mathcal{I}_1[F]+\mathcal{I}_2[F]$.\\
As we will see, the problem is in the $\mathcal{I}_1$ functional, since we are only able to prove weak lower-semicontinuity of this functional on the space $L^p(0,T)$, which is quite larger than $\cA$.
The proof of the next Lemma mimics the one of \cite[Theorem $3.20$]{Dacorogna}.
\begin{lem}
	Consider $p \ge 1$ and suppose we have a sequence $F_n \in L^p(0,T)$, a function $F \in L^p(0,T)$ such that $F_n \rightharpoonup F$ in $L^p$. If $x \mapsto \E[J(t,Z(t),x)]$ is convex $\forall t \in (0,T)$, then
	\begin{equation*}
	\liminf_{n}\mathcal{I}_1[F_n]\ge \mathcal{I}_1[F].
	\end{equation*}
\end{lem}
\begin{proof}
	First of all, let us observe that since $J$ is continuous in $x$, then if $F_n \to F$ in $L^p$ we have, by Fatou's Lemma,
	\begin{align*}
	\liminf_{n}\mathcal{I}_1[F_n]&=\liminf_{n}\E\left[\int_0^T J(t,Z(t),F_n(t))dt\right] \\
	&\ge \E\left[\int_0^T \liminf_{n} J(t,Z(t),F_n(t))dt\right]=\mathcal{I}_1[F]
	\end{align*}
	so in particular $\mathcal{I}_1$ is strong lower semicontinuous.\\
	Now, if $\liminf_{n}\mathcal{I}_1[F_n]=+\infty$ the theorem is trivial. Suppose then $\liminf_{n}\mathcal{I}_1[F_n]=C<+\infty$ and suppose we are working with a subsequence (that, for the ease of the reader, we will still call $F_n$) such that $\lim_n \mathcal{I}_1[F_n]=C$. Fix $\varepsilon>0$ and observe that there exists a $\nu_\varepsilon$ such that for $n \ge \nu_\varepsilon$ we have $\mathcal{I}_1[F_n]\le C+\varepsilon$. Now, by Mazur's Theorem \cite[Theorem $3.9$]{Dacorogna} we know that there exists a sequence of integers $\{m_\mu\}_{\mu \in \N}$ with $m_\mu\ge \nu_\varepsilon$ and for each $\mu \in \N$ a vector $a_\mu \in \R^{m_\mu-\nu_\varepsilon}$ with $\sum_{i=1}^{m_\mu-\nu_\varepsilon}a_\mu^i=1$ such that, if we pose
	\begin{equation*}
	G_\mu=\sum_{i=1}^{m_\mu}a_\mu^i F_{i+\nu_\varepsilon},
	\end{equation*}
	we have $G_\mu \to F$ in $L^p$. However, by convexity of $\E[J(t,Z(t),x)]$ in $x$ (by also using Fubini's theorem) we have
	\begin{equation*}
	\mathcal{I}_1[G_\mu]\le \sum_{i=1}^{m_\mu}a_\mu^i\mathcal{I}_1[F_{i+\nu_\varepsilon}]\le C+\varepsilon.
	\end{equation*}
	Taking the $\liminf$ on $\mu$, from the strong lower semi-continuity, we have
	\begin{equation*}
	\mathcal{I}_1[F]\le \liminf_{\mu \to +\infty} \mathcal{I}_1[G_\mu] \le C+\varepsilon
	\end{equation*}
	Finally, we can send $\varepsilon \to 0$ to conclude.
\end{proof}
The latter result shows us that if we want to use an approach via minimizing sequences to find a minimizer, we can do this by substituting $L^p(0,T)$ for some $p \ge 1$ to $\cA$. However, $L^p$ functions are not defined on single points, thus for any $F \in L^p(0,T)$, $F(T)$ is not well-defined. Hence we need to split the action of $\cI_1$ and $\cI_2$, the first on $L^p(0,T)$, the second simply on $\R$. From now on, our admissible set will be composed of couples $(F,a)$ where $F \in L^p(0,T)$ and $a \in \R$. Let us define the relaxed admissible set
\begin{equation*}
\widetilde{\mathcal{A}}_p=\{(F,a)\in L^p(0,T)\times \R\}
\end{equation*}
and the relaxed functional
\begin{equation*}
\widetilde{\mathcal{J}}[(F,a)]=\mathcal{I}_1[F]+\mathcal{I}_2[a]
\end{equation*}
for $(F,a)\in \widetilde{\mathcal{A}}_p$. Then the relaxed problem is given by
\begin{equation}\label{rprob}
\mbox{ find }\arg\min_{(F,a)\in \widetilde{\mathcal{A}}_p}\widetilde{\mathcal{J}}[(F,a)].
\end{equation}
Now we can move to the next step, that is proving that the relaxed problem \eqref{rprob} admits a solution.
\begin{lem}
Under Assumption $A1-A9$, Problem \eqref{rprob} admits a solution.
\end{lem}
\begin{proof}
	Let us first consider the case in which $x \mapsto \E[\Phi(Z(T),x)]$ is a proper map. If $\inf_{(F,a)\in \widetilde{\mathcal{A}}_p} \widetilde{\mathcal{J}}[(F,a)]=+\infty$, the solution is trivial. Thus let us suppose $\inf_{(F,a)\in \widetilde{\mathcal{A}}_p} \widetilde{\mathcal{J}}[(F,a)]=L\ge 0$. Let us then consider a sequence $(F_n,a_n) \in \widetilde{\mathcal{A}}_p$ such that $\lim_{n}\widetilde{\mathcal{J}}[(F_n,a_n)]=L$. In particular we can suppose that $$\widetilde{\mathcal{J}}[(F_n,a_n)]<L+1, \  \forall n \in \N.$$
	First of all, we have that
	\begin{equation*}
	\mathcal{I}_1[F_n]=\E\left[\int_0^T J(t,Z(t),F_n(t))dt\right]<L+1.
	\end{equation*}
	By Fubini's Theorem we have that
	\begin{equation*}
	\int_0^T \E[J(t,Z(t),F_n(t))]dt<L+1.
	\end{equation*}
	Now let us define $\mathcal{M}_n=\{t \in [0,T]: \ |F_n(t)|>M\}$. By using Assumption $A8$ we have
	\begin{multline*}
	\alpha \left(\int_{\mathcal{M}_n}h(t)dt+\int_{\mathcal{M}_n}|F_n(t)|^pdt\right)\le \\ \le \int_{\mathcal{M}_n} \E[J(t,Z(t),F_n(t))]dt\le \\ \le \int_0^T \E[J(t,Z(t),F_n(t))]dt<L+1
	\end{multline*}
	and then
	\begin{equation}\label{est1}
	\int_{\mathcal{M}_n}|F_n(t)|^pdt<\frac{L+1}{\alpha}-\int_{\mathcal{M}_n}h(t)dt\le \frac{L+1}{\alpha}+\Norm{h}{L^1(0,T)}.
	\end{equation}
	At the same time we have
	\begin{equation}\label{est2}
	\int_{[0,T]\setminus \mathcal{M}_n}|F(t)|^pdt\le TM^p.
	\end{equation}
	Thus we have, by summing Equations \eqref{est1} and \eqref{est2}
	\begin{equation*}
	\Norm{F_n}{L^p(0,T)}^p<\frac{L+1}{\alpha}+\Norm{h}{L^1(0,T)}+TM^p, \ \forall n \in \N
	\end{equation*}
	and then, by Banach-Alaoglu theorem \cite[Theorem $3.16$]{brezis2010functional} (and the fact that $L^p$ is reflexive \cite[Theorem $4.10$]{brezis2010functional}), there exists a $F \in L^p(0,T)$ such that (up to a subsequence) $F_n \rightharpoonup F$ in $L^p$.\\
	Moreover, by the weak lower-semicontinuity of $\mathcal{I}_1$ we have
	\begin{equation*}
	\mathcal{I}_1[F]\le \liminf_{n}\mathcal{I}_1[F_n].
	\end{equation*}
	Now, we also have that
	\begin{equation*}
	\E[\Phi(Z(T),a_n)]\le L+1, \ \forall n \in \N
	\end{equation*}
	thus, since $x \mapsto \E[\Phi(Z(T),x)]$ is a proper map by Assumption $A9$, there exists $M>0$ such that $|a_n|\le M$. Thus there exists $a \in \R$ such that (up to a subsequence) $a_n \to a$. Moreover, we have that for any $n \in \N$, posing $K=[-M,M]$
	\begin{equation*}
	\Phi(Z(T),a_n)\le 2M\Theta_K(Z(T))+\Phi(Z(T),a_1)
	\end{equation*}
	hence we can use dominated convergence theorem to conclude that
	\begin{equation*}
	\lim_{n}\mathcal{I}_2[a_n]=\mathcal{I}_2[a].
	\end{equation*}
	Finally, we have that
	\begin{multline*}
	\widetilde{\mathcal{J}}[(F,a)]=\mathcal{I}_1[F]+\mathcal{I}_2[a]\le \\ \le \liminf_{n \in \N}\mathcal{I}_1[F_n]+\lim_{n}\mathcal{I}_2[a_n]=\\=\liminf_{n}\widetilde{\mathcal{J}}[(F_n,a_n)]=L.
	\end{multline*}
	If the map $x \mapsto \E[\Phi(Z(T),x)]$ is constant, the term in $\Phi$ is actually a \textit{dummy} term and it is useless in the minimization problem. Hence we can neglect it and the statement still holds true.
\end{proof}
\begin{rmk}
	We did not use $A7$ but only the fact that for any fixed $t \in [0,T]$ the map $x \mapsto \E[J(t,Z(t),x)]$ is convex. The final formulation of $A7$ is needed to guarantee that if $(F,a)$ and $(\widetilde{F},b)$ are solutions, then $F=\widetilde{F}$ in $L^p(0,T)$.
\end{rmk}
Let us also recall that, with the same strategy used before, we can show the following necessary optimality conditions for the minima of the relaxed problem.
\begin{lem}\label{thmnecr}
	Let $(F^*,a^*) \in \widetilde{\cA}_p$ be a solution of the problem \eqref{rprob}. Then, under assumptions $A1-A6$,
	\begin{equation}\label{eulanr}
	\E\left[\pd{J}{x}(t,Z(t),F^*(t))\right]=0 	\qquad \mbox{for almost all } t \in (0,T)
	\end{equation}
	and
	\begin{equation}\label{transvr}
	\E\left[\pd{\Phi}{x}(Z(T),a^*)\right]=0.
	\end{equation}
\end{lem}
\subsection{Gain of regularity}\label{Sec.3.4}
The solutions we found for the relaxed problem at this stage cannot be used for the original problem:
\begin{itemize}
\item $F^* \in L^p(0,T)$ while we want it in $AC$;
\item $F^*(T)$ is not well-defined, but, even if it were, we are not sure that, in any case, $F^*(T)=a^*$.
\end{itemize}
For these reasons we have to show that we can gain regularity of the solution, in the sense that we can find a solution $F^*$ of the relaxed problem \eqref{rprob} that is more regular than simply $L^p(0,T)$ for some $p \in (1,+\infty)$.\\
First of all, let us show that, under our assumptions, the first part of the solution $F^*$ is unique and continuous, while $a^*$ can only vary in an interval.
\begin{lem}\label{propcont}
	Under Assumptions $A1-A10$, there exists a unique $F^* \in L^p(0,T)$ and a unique interval $I \subseteq \R$ such that for any $a^* \in I$ the couple $(F^*,a^*) \in \widetilde{\cA}_p$ is a solution of the problem \eqref{rprob}. Moreover $F^*$ admits a continuous modification in $[0,T]$ for which the equation \eqref{eulanr} holds for any $t \in [0,T]$.
\end{lem}
\begin{proof}
	Since the map $x \mapsto \E[\Phi(Z(T),x)]$ is convex by Assumption $A10$, we know that \linebreak $\argmin_{x \in \R}\E[\Phi(Z(T),x)]$ must be a convex set, hence, being non-empty by the previous theorem, it must be an interval $I$. Moreover the strict convexity of the map $x \mapsto \E[J(t,Z(t),x)]$ given by Assumption $A7$ ensures that the minimizer $F^* \in L^p(0,T)$ (that exists by the previous theorem) is unique.\\
	Concerning the continuity, fix $t_0 \in [0,T]$ and consider $t_n \to t_0$ such that $t_n$ are points for which Equation \eqref{eulanr} is satisfied. We obviously have
	\begin{equation*}
	\lim_{n \to +\infty}\E\left[\pd{J}{x}(t_n,Z(t_n),F^*(t_n))\right]=0.
	\end{equation*}
	Now observe that  there exists a subsequence of $F^*(t_n)$ that converges to  $\liminf_{n}F^*(t_n)$. In particular, taking the limit on such subsequence, we have, by Assumption $A3$,
	\begin{equation*}
	\E\left[\pd{J}{x}(t_0,Z(t_0),\liminf_{n}F^*(t_n))\right]=0.
	\end{equation*}
	In the same way, we also have
	\begin{equation*}
	\E\left[\pd{J}{x}(t_0,Z(t_0),\limsup_{n}F^*(t_n))\right]=0.
	\end{equation*}
	Now, since $x \mapsto \E[J(t,Z(t),x)]$ is strictly convex for any $t \in [0,T]$, we know that $x \mapsto \E\left[\pd{J}{x}(t,Z(t),x)\right]$ is injective, thus for any $t \in [0,T]$ the equation
	\begin{equation}\label{eqnec}
	\E\left[\pd{J}{x}(t_0,Z(t_0),x)\right]=0
	\end{equation}
	admits a unique solution and then $\liminf_{n}F^*(t_n)=\limsup_{n}F^*(t_n)$. We have shown that $\lim_n F^*(t_n)$ is well-defined. Now let us observe that being $x \mapsto \E[J(t,Z(t),x)]$ decreasing as $x \to -\infty$ and increasing as $x \to +\infty$, $\lim_n F^*(t_n)\not = \pm \infty$ by Assumption $A7$. Thus we have $\lim_{n}F^*(t_n)\in \R$. Now let us distinguish two cases. If $t_0 \in [0,T]$ is one of the point in which the necessary condition \eqref{eulanr} is already satisfied, we have, by uniqueness of the solution of Equation \eqref{eqnec}, $F^*(t_0)=\lim_{n}F^*(t_n)$ and then $F^*$ is continuous in $t_0$.\\
	If $t_0$ is not one of these points, we can modify $F^*$ on $t_0$ in such a way that $F^*(t_0)=\lim_{n}F^*(t_n)$. Being the set of $t_0 \in [0,T]$ for which the necessary condition is not satisfied a zero-measure set, we can conclude that $F^*$ admits a continuous modification in $[0,T]$.
\end{proof}
It is still not enough: we do not want $F^*$ to be simply continuous, but absolutely continuous. However, under our hypotheses we do not only obtain that $F^*$ is absolutely continuous, but we can exploit its derivative (almost everywhere).
\begin{lem}\label{propabscont}
	Under Assumptions $A1-A12$, let $(F^*,a^*)$ be a solution of \eqref{rprob}. Then $F^* \in AC[0,T]$.
\end{lem}
\begin{proof}
	By the Implicit Function Theorem (see \cite[Theorem $3.2.1$]{krantz2012implicit}) we know that $\eta(t)$ defined in Equation \eqref{eta} of Assumption $A12$ is actually the derivative of $F^*(t)$ where it is defined. In particular let us denote $\cZ=\{t_1,\dots,t_n\}$, $t_0=0$, $t_{n+1}=T$ and $I_j=(t_{j-1},t_j)$ for $j=1,\dots,n+1$. For any $j=1,\dots,n+1$ and $t \in I_j$ we have $\eta(t)=\der{F^*}{t}(t)$ and, being $\eta$ continuous in such interval, $F^* \in C^1(I_j)$ for any $j=1,\dots,n+1$. Fix a $j \in \{1,\dots,n+1\}$. Let us observe that for any $\varepsilon>0$
	\begin{equation*}
	\int_{t_{j-1}+\varepsilon}^{t_j-\varepsilon}\eta(s)ds=F^*(t_j-\varepsilon)-F^*(t_{j-1}+\varepsilon)
	\end{equation*}
	thus, by the dominated convergence theorem (being $\eta$ in $L^1$) and by the continuity of $F$ we have, by taking $\varepsilon \to 0$,
	\begin{equation}\label{case1}
	\int_{t_{j-1}}^{t_j}\eta(s)ds=F^*(t_j)-F^*(t_{j-1}).
	\end{equation}
	Moreover, if we consider $t \in I_j$ we can show in the same way that
	\begin{equation}\label{case2}
	\int_{t_{j-1}}^{t}\eta(s)ds=F^*(t)-F^*(t_{j-1}).
	\end{equation}
	Now let us consider $t \in [0,T]$. If $t=t_j$ for some $j=\{1,\dots,n+1\}$ we have, by Equation \eqref{case1},
	\begin{equation*}
	\int_0^{t_j}\eta(s)ds=\sum_{k=1}^{j}\int_{t_{k-1}}^{t_k}\eta(s)ds=\sum_{k=1}^{j}F^*(t_k)-F^*(t_{k-1})=F^*(t_j)-F^*(0).
	\end{equation*}
	Otherwise there exists $j \in\{1,\dots,n+1\}$ such that $t \in I_j$ and we have, by both Equations \eqref{case1} and \eqref{case2},
	\begin{multline*}
	\int_0^{t}\eta(s)ds=\sum_{k=1}^{j-1}\int_{t_{k-1}}^{t_k}\eta(s)ds+\int_{t_{j-1}}^{t}\eta(s)ds=\\=\sum_{k=1}^{j-1}F^*(t_k)-F^*(t_{k-1})+F^*(t)-F^*(t_{j-1})=F^*(t)-F^*(0).
	\end{multline*}
	Thus, for any $t \in [0,T]$, we have
	\begin{equation*}
	F^*(t)=F^*(0)+\int_0^t\eta(s)ds
	\end{equation*}
	concluding the proof.
\end{proof}
\begin{rmk}\label{RemBM} Let us stress that Assumption $A11$ can be lightened by asking instead that $\pdsup{}{x}{2}\E[J(t,Z(t),F^*(t))]>0$ except for a set $\mathcal{Z}$ that is at most finite.
\end{rmk}
Now, for $F^*$ to be in $\cA$, we only need to ask that $F(0)=0$ and $F(T) \in I$, where $I$ is the optimal interval for $\cI_2$. This is done by introducing the Assumptions $A13-A14$.
\begin{lem}
Under Assumptions $A1-A14$, \eqref{prob} admits a unique solution $F \in \cA$.
\end{lem}
\begin{proof}
	Let us consider first the relaxed problem. Thus we have that there exists a unique function $F \in L^p(0,T)$ and an interval $I \subseteq\R$ such that for any $a \in I$ the couple $(F,a) \in \widetilde{\cA}_p$ is solution of \eqref{rprob}.  By Lemma \ref{propcont} we know that $F \in C^0([0,T])$. Moreover, by Lemma \ref{propabscont} we know that $F \in AC([0,T])$. By Assumption $A13$ and uniqueness of the solution of equation \eqref{eqnec} we have that $F(0)=0$, hence $F \in \cA$. Finally, we have that $F(T)$ is the unique solution of \eqref{eulan} for $t=T$ and then, by Assumption $A14$,
	\begin{equation*}
	\E\left[\pd{\Phi}{x}(Z(T),F(T))\right]=0.
	\end{equation*}
	Being the map $x \mapsto \E[\Phi(Z(t),x)]$ convex, the map $x \mapsto \E\left[\pd{\Phi}{x}(Z(T),x)\right]$ is increasing and then $F(T) \in I$. Thus the couple $(F,F(T))$ is solution of the relaxed Problem \eqref{rprob} and then $F$ is solution of Problem \eqref{prob}. Uniqueness follows from the fact that for each $t>0$, Equation \eqref{eulanr} admits a unique solution.
\end{proof}
This last Lemma ends the proof of Theorem \ref{thmmain}. The only thing we have to observe is that by implicit function theorem and Assumptions $A11-A12$, we know that Equation \eqref{eulan} admits a unique solution that, in such case, has to be the minimizer we are looking for.

\subsection{Power cost functionals}\label{Sec.3.7}
Let us give a practical example. We want to solve the approximation problem of Section \ref{Sec.3} for some particular cost functions. By power cost functionals we mean functionals $\cJ_p$ induced by $\Phi$ constant and, for a fixed $p \ge 2$, for any $t \in [0,T]$ and $X^f \in OU(a,\sigma,x_0)$,
\begin{equation}\label{powercost}
J_p(t,|X(t)-X^f(t)|)=|X(t)-X^f(t)|^{p}.
\end{equation}
 We can show the following result that will be useful in the applications.
\begin{prop}\label{prop3.10}
Let us fix $p>2$ and suppose that
\begin{itemize}
	\item[$i$] The process $z \in \cL^p(\Omega,\{\mathcal{F}_t\}_{t\geq 0};L^1(0,T))$;
	\item[$ii$] The function $t \mapsto \E[|z(t)|^{p}]$ belongs to $L^1(0,T)$.
	\item[$iii$] $z(t)>0$ almost surely for any $t \in (0,T]$;
	\item[$iv$] $\pd{}{t}\E[|x-Z(t)|^{p-2}(x-Z(t))]$ is continuous.
\end{itemize}
Fix $\Phi \equiv C$ for some constant $C\ge 0$ and let $J_p$ be as in Equation \eqref{powercost}.
Then the Problem \eqref{prob} with running cost function \eqref{powercost} and constant final cost admits a unique solution $F_p \in \cA$. The same holds if $\Phi_p(X(T)-X^f(T))=|X(T)-X^f(T)|^p$.\\
If $p=2$, then, under hypotheses $i$ and $ii$, for $\Phi \equiv C$ or $\Phi_2(X(T)-X^f(T))=|X(T)-X^f(T)|^2$ and $J_2$ defined as before, there exists a unique $F_2 \in \cA$ solution of \eqref{prob}. 
\end{prop}
\begin{proof}
Without loss of generality, let us always suppose that $C=0$. As function of $(t,z,x)\in R_T \times \R$ we have, by Equation \eqref{powercost},
\begin{equation*}
J_p(t,z,x)=|z-x|^{p}.
\end{equation*}
Let us also denote the respective functional as $\cJ_p$. By definition of $Z(t)$ in Equation \eqref{Equ.2.3} and Jensen's inequality we have
\begin{equation*}
\E[|Z(t)|^p]\le e^{pA(t)}\int_0^t\E[|z(t)|^p]e^{-pA(s)}ds
\end{equation*}
and in particular, since the right-hand side is continuous, we have that $\E[|Z(t)|^p]\in L^1(0,T)$. 
Let us now check the hypotheses of Theorem \ref{thmmain}.
\begin{itemize}
	\item[A1] This is verified for $F\equiv 0$, since $\cJ_p[F]=\int_0^T \E[|Z(t)|^p]dt<+\infty$;
	\item[A2] This hypothesis is verified by definition of $J_p$ and $\Phi$;
	\item[A3] For any fixed $(t,z)\in R_T$ the map $x \mapsto J_p(t,z,x)$ belongs to $C^1$ and
	\begin{equation*}
	\pd{J_p}{x}(t,z,x)=p|x-z|^{p-2}(x-z);
	\end{equation*}
	\item[A4] Let $x \in [-K,K]$ and observe that
	\begin{equation*}
	\left|\pd{J_p}{x}(t,z,x)\right|=p|x-z|^{p-1}\le p 2^{p-2}(K^{p-1}+|z|^{p-1})=:\Psi_K(z).
	\end{equation*}
	Since $\E[|Z(t)|^{p-1}]$ is well defined and belongs to $L^1(0,T)$ (by H\"older's inequality), we have that $\Psi_K(z)\in L^1(R_T;\mu)$.
	\item[A5,A6] These are obvious since $\Phi \equiv 0$;
	\item[A7] The map $x \mapsto \E[J(t,Z(t),x)]$ is strictly convex, decreasing as $x \to -\infty$ and increasing as $x \to +\infty$ since the function $J(t,z,x)$ has these properties;
	\item[A8] Observe that we have
	\begin{equation*}
	|x|^p=|x-z+z|^p\le 2^{p-1}(|x-z|^p+|z|^p)
	\end{equation*}
	hence
	\begin{equation*}
	|x-z|^p\ge 2^{1-p}|x|^p-|z|^p.
	\end{equation*}
	Thus we can conclude that
	\begin{equation*}
	\E[J_p(t,Z(t),x)]\ge 2^{1-p}(|x|^p-2^{p-1}\E[|Z(t)|^p]),
	\end{equation*}
	where $2^{p-1}\E[|Z(t)|^p]$ belongs to $L^1(0,T)$;
	\item[A9,A10] These hypotheses are obviously satisfied by $\Phi \equiv 0$;
	\item[A11] Let us observe that
	\begin{equation*}
	\pdsup{}{x}{2}J_p(t,z,x)=p(p-1)|x-z|^{p-2}.
	\end{equation*}
	Let us fix $x_0 \in \R$, $\delta>0$ and distinguish three cases. If $p\ge 3$ then $p-2>1$ and we have
	\begin{multline*}
	p(p-1)|x-z|^{p-2}\le p(p-1)2^{p-2}(|x|^{p-2}+|z|^{p-2})\le \\
	\le p(p-1)2^{p-2}(\max\{|x_0-\delta|^{p-2},|x_0+\delta|^{p-2}\}+|z|^{p-2})\\
	\quad \forall x \in [x_0-\delta,x_0+\delta].
	\end{multline*}
	If $2<p<3$ and $xz \ge 0$ we have (since the function $|x|^{p-2}$ is concave if restricted to $(-\infty,0]$ or $[0,+\infty)$)
	\begin{equation*}
	|z|^{p-2}=|z-x+x|^{p-2}\ge 2^{p-3}(|x-z|^{p-2}+|x|^{p-2})
	\end{equation*}
	thus in this case
	\begin{equation*}
	p(p-1)|x-z|^{p-2}\le p(p-1)(2^{3-p}|z|^{p-2}-|x|^{p-2})\le p(p-1)2^{3-p}|z|^{p-2}.
	\end{equation*}
	If $2<p<3$ and $xz<0$ then we can suppose $x>0$ and $z<0$. In this case
	\begin{align*}
	p&(p-1)|x-z|^{p-2}=p(p-1)|x+|z||^{p-2}\\&\le p(p-1)(|x|+|z|)^{p-2}\\&\le 2^{p-2}p(p-1)(\max\{|x|,|z|\})^{p-2}\\
	&\le 2^{p-2}p(p-1)(\max\{|x_0-\delta|,|x_0+\delta|,|z|\})^{p-2} \quad \forall x \in [x_0-\delta,x_0+\delta].
	\end{align*}
	The same holds for $x<0$ and $z>0$. Finally, for $p=2$ we have $p(p-1)|x-z|^{p-2}=2$.
	By using these estimates and the estimate in hypothesis $A4$, we can differentiate under the integral sign, obtaining
	\begin{equation*}
	\pdsup{}{x}{2}\E[J_p(t,Z(t),x)]=p(p-1)\E[|x-Z(t)|^{p-2}]\ge 0.
	\end{equation*}
	Let us observe that, by Remark \eqref{RemBM}, we actually need to show that $\E[|F_p(t)-Z(t)|^{p-2}]=0$ at most in a finite set, where $F_p(t)$ is the unique solution of
	\begin{equation}\label{eulanpower}
	\E\left[|x-Z(t)|^{p-2}(x-Z(t))\right]=0 \qquad t \in (0,T),
	\end{equation}
that is Equation \eqref{eulan} in this case.
Let us first consider the case in which $p>2$. Since, by $iii$, $z(t)>0$ for any $t \in [0,T]$ almost surely, $Z(t)>0$ for any $t$ in $(0,T]$ almost surely. In such case, $F_p(t)$ cannot be negative for all $t \in [0,T]$ thus $\max_{t \in [0,T]}F_p(t)>0$. Then we have, recalling Equation \eqref{eulanpower},
	\begin{align*}
	0=\E[|F_p(t)-Z(t)|^{p-2}&(F_p(t)-Z(t))]<\E\left[|F_p(t)-Z(t)|^{p-2}F_p(t)\right]\\&<\max_{t \in [0,T]}F_p(t)\E\left[|F_p(t)-Z(t)|^{p-2}\right].
	\end{align*}
	For $p=2$ we have instead
	\begin{equation*}
	\pdsup{}{x}{2}\E[J_2(t,Z(t),x)]=2>0,
	\end{equation*}
	without using $iii$.
	\item[A12] For $p>2$, let us observe that
	\begin{equation*}
	\pdmis{}{x}{t}\E[J_p(t,Z(t),x)]=p\pd{}{t}\E[|x-Z(t)|^{p-2}(x-Z(t))].
	\end{equation*}
	Since, by $iv$, we have that $\pd{}{t}\E[|x-Z(t)|^{p-2}(x-Z(t))]$ is continuous in $t \in [0,T]$, we can observe that the function
	\begin{equation*}
	\eta_p(t)=\frac{\pd{}{t}\E[|F_p(t)-Z(t)|^{p-2}(Z(t)-F_p(t))]}{(p-1)\E[|F_p(t)-Z(t)|^{p-2}]}
	\end{equation*}
	is continuous in $[0,T]$, thus is in $L^1(0,T)$.\\
The case $p=2$ is simpler. We have that Equation \eqref{eulanpower} becomes
\begin{equation*}
\E[Z(t)-F_2(t)]=0
\end{equation*}
thus we know that 
\begin{equation}\label{F2}
F_2(t)=\E[Z(t)].
\end{equation} 
By Fubini's theorem we know that
\begin{equation*}
\E[Z(t)]=e^{A(t)}\int_{0}^{t}\E[z(s)]e^{-A(s)}ds.
\end{equation*}
Thus we have, since $F_2(t)=e^{A(t)}\int_0^{t}f_2(s)e^{-A(s)}ds$,
\begin{equation}\label{f2}
f_2(s)=\E[z(s)]
\end{equation}
which is uniquely defined since the map $\cI$ is a bijection. Finally, since we have 
\begin{equation*}
\der{}{t}\E[Z(t)]=a(t)\E[Z(t)]+\E[z(t)]
\end{equation*}
and then
\begin{equation*}
\pdmis{}{x}{t}\E[(x-Z(t))^2]=-2(a(t)\E[Z(t)]+\E[z(t)]),
\end{equation*}
we get
\begin{equation*}
\eta(t)=a(t)\E[Z(t)]+\E[z(t)]
\end{equation*}
that is continuous in $[0,T]$ and then in $L^1(0,T)$.
	\item[A13] Since $Z(0)=0$ we have
	\begin{equation*}
	\E\left[\pd{J_p}{x}(0,0,0)\right]=0.
	\end{equation*}
	\item[A14] This hypothesis is obviously satisfied since $\Phi \equiv 0$.	
\end{itemize}
The proof of Assumptions $A4$,$A5$ for $\Phi_p$ are analogous to the ones for $J_p$, while $A9$, $A10$ and $A14$ follow from the structure of $\Phi_p$. 
\end{proof}
The previous example shows that our result includes, as a special case, the well-known fact that the expected value minimizes the mean square error.
Moreover, in this case it is easy to obtain a bound on the minimum in terms of the variance of $z(t)$. Indeed we have
\begin{equation*}
\cJ_2[\E[Z(t)]]=\E\left[\int_{0}^{T}(Z(t)-\E[Z(t)])^2dt\right]=\int_0^{T}\D[Z(t)]dt.
\end{equation*}
However we have by Jensen's inequality
\begin{equation}\label{d2def}
\D[Z(t)]\le e^{2A(t)}\int_0^t\D[z(s)]e^{-2A(s)}ds=:d_2(t)
\end{equation}
thus we have
\begin{equation*}
\cJ_2[\E[Z(t)]]\le \int_0^{T}d_2(t)dt.
\end{equation*}
In particular, if $d_2 \in L^1(0,+\infty)$, we have that the minimum approximation error $\cJ_2[\E[Z(t)]]$ is bounded for $T \to +\infty$.\\
One can obtain also a point-wise estimate on the distance between the processes, given by
\begin{equation}\label{pointdist}
\E[|X(t)-X^f(t)|^2]\le d_2(t).
\end{equation}
Let us recall that, despite Problem \eqref{prob} admits a unique solution, the relaxed Problem \eqref{rprob} could still admit more than a solution. For instance, if we consider the functional $\mathcal{J}_2$ induced by $J_2$ and $\Phi \equiv 0$, Problem \eqref{prob} admits $F_2(t)=\E[Z(t)]$ as unique solution, while Problem \eqref{rprob} admits $(F_2(t),a)$ as solution for any $a \in \R$. On the other hand, if $\Phi$ is strictly convex, then also the relaxed Problem \eqref{rprob} admits a unique solution. This is the case of the functional $\mathcal{J}_{2,2}$ induced by $J_2$ and $\Phi_2(x,z)=|x-z|^2$. Indeed Problem \eqref{prob} admits $F_2(t)=\E[Z(t)]$ as unique solution and Problem \eqref{rprob} admits $(F_2(t),F_2(T))$ as unique solution. Finally, let us observe that if Assumption $A14$ is not satisfied, then the Problem \eqref{rprob} could admit a solution while \eqref{prob} could not. Indeed, if we consider the funcitonal $\mathcal{J}_{2,3}$ induced by $J_2$ and $\Phi_3(x,z)=|x-z|^3$, Problem \eqref{rprob} admits as unique solution $(F_2(t),F_3(T))$, while Problem \eqref{prob} admits a unique solution if and only if $F_2(T)=F_3(T)$, otherwise there are no absolutely continuous functions that satisfies simultaneously Equations \eqref{eulan} and \eqref{transv}.

\section{Examples}\label{Sec.4}

In order to provide some examples of application of Theorem \ref{thmmain}, in particular of Proposition \ref{prop3.10}, let us consider Eq. \eqref{Equ.2.1} with $a(t)\equiv-\theta$ for some $\theta>0$.
Let us study the optimal approximation for some particular choices of $z(t)$. In particular we will denote with $X(t)$ the original process and with $X_p(t)$ the optimal Gauss-Markov approximation with respect to the power cost functional $\cJ_p$. In the examples we will consider the approximations $X_2(t)$ and $X_4(t)$. Let us recall that, by Proposition \ref{prop3.10}, $F_2(t)=\E[Z(t)]$ while Equation \eqref{eulanpower} for $p=4$ becomes
\begin{equation}\label{eulanpower4}
F_4(t)^3-3F_4(t)^2\E[Z(t)]+3F_4(t)\E[Z(t)^2]-\E[Z(t)^3]=0.
\end{equation}
Let us also recall that the approximating process $X_p(t)$ solves the SDE
\begin{equation}\label{SDEapproxp}
dX_p(t)=[a(t)X_p(t)+f_p(t)]dt+\sigma dW(t), \quad X_p(0)=x_0
\end{equation}
where 
\begin{equation}\label{fpcalc}
f_p(t)=\mathcal{I}^{-1}F_p(t)=\eta_p(t)-a(t)F_p(t).
\end{equation}
$F_4(t)$ is, by Theorem \ref{thmmain} and Proposition \ref{prop3.10}, the unique zero of Equation \eqref{eulanpower4}. Thus we can evaluate it numerically by using bisection method. We will not have explicit expression of $X_4(t)$ in the following examples.
\subsection{A single shot as a drift.}\label{Sec.4.1}
Let 
\begin{equation*}
z(t)=\begin{cases}
0 & t<\cT\\
1 & t \ge \cT
\end{cases}
\end{equation*}
where $\cT \sim Exp(\lambda)$ with $\lambda \not = \theta,2\theta$. To ensure that $z(t)$ is adapted to the filtration $\{\cF_t\}_{t \ge 0}$ we have to ask for $\cT$ to be a Markov time with respect to that filtration. By definition of $Z(t)$ in Equation \eqref{Equ.2.3}, we have
\begin{equation}\label{Zss}
Z(t)=\frac{(1-e^{-\theta(t-\cT)})}{\theta}\chi_{[\cT,+\infty)}(t),
\end{equation}
where $\chi_{[\cT,+\infty)}(t)$ is the indicator function of the (stochastic) interval $[\cT,+\infty)$. Let us in particular observe that $Z(t)$ is a Markov process. Moreover, by using Proposition \ref{prop2.1} (or Equation \eqref{spliteq}) we have
\begin{equation*}
X(t)=e^{-\theta t}x_0+\sigma e^{-\theta t}\int_0^{t}e^{\theta s}dW_s+\frac{(1-e^{-\theta(t-\cT)})}{\theta}\chi_{[\cT,+\infty)}(t).
\end{equation*}
Let us also observe that
\begin{equation*}
f_2(t)=\E[z(t)]=1-e^{-\lambda t},
\end{equation*}
hence we obtain
\begin{equation}\label{F2ss}
F_2(t)=\cI f_2(t)=\frac{1-e^{-\theta t}}{\theta}-\frac{e^{-\lambda t}-e^{-\theta t}}{\theta-\lambda}.
\end{equation}
Finally, $X_2(t)$ is obtained by solving Equation \eqref{SDEapproxp}, thus we have 
\begin{equation*}
X_2(t)=\cS f_2(t)=e^{-\theta t}x_0+\sigma e^{-\theta t}\int_0^{t}e^{\theta s}dW_s+\frac{1-e^{-\theta t}}{\theta}-\frac{e^{-\lambda t}-e^{-\theta t}}{\theta-\lambda}.
\end{equation*}
Observing that
\begin{equation*}
\D[z(t)]=(1-e^{-\lambda t})-(1-e^{-\lambda t})^2
\end{equation*}
we have, by Equation \eqref{d2def},
\begin{equation}\label{ssd2}
d_2(t)=\frac{e^{-\lambda t}((\lambda-2\theta)e^{-\lambda t}+2\theta-2\lambda)+\lambda e^{-2\theta t}}{2(\theta-\lambda)(2\theta-\lambda)}.
\end{equation}
It is not difficult to check that $d_2 \in L^1(0,+\infty)$ thus $\cJ_2[X_2]$ is uniformly bounded by $\Norm{d_2}{L^1(0,+\infty)}$.
To evaluate $X_4(t)$, we need $F_4(t)$, solution of Equation \eqref{eulanpower4}. In Figure \ref{fig:f2f4}, on the left, we plot some sample trajectories of $X(t)$, $X_2(t)$ and $X_4(t)$, with the same realization of the white noise process. In particular to plot $X_4(t)$ we solved numerically Equation \eqref{eulanpower4} to obtain $F_4(t)$. In Figure \ref{fig:f2f4}, on the right, we plot a simulated curve of the function $t \mapsto \E[(X(t)-X_2(t))^2]$ and we compare it to the bound $d_2(t)$ given in Equation \eqref{ssd2}. We observe that as $t$ increases, the simulated error and the bound tend to coincide and go to $0$. The fact that the error should go to $0$ can be also observed from Figure \ref{fig:f2f4} on the left, since as $t$ increases the trajectories of $X_2$ and $X$ overlap. This is due also by the nature of $Z(t)$ in Equation \eqref{Zss}, which goes to $1/\theta$ as $t$ increases and so does its mean, thus leaving the stochastic part only to $Y(t)$ that is common between $X_2$ and $X$.
\begin{figure}[h]
	\centering
	\includegraphics[width=0.49\linewidth]{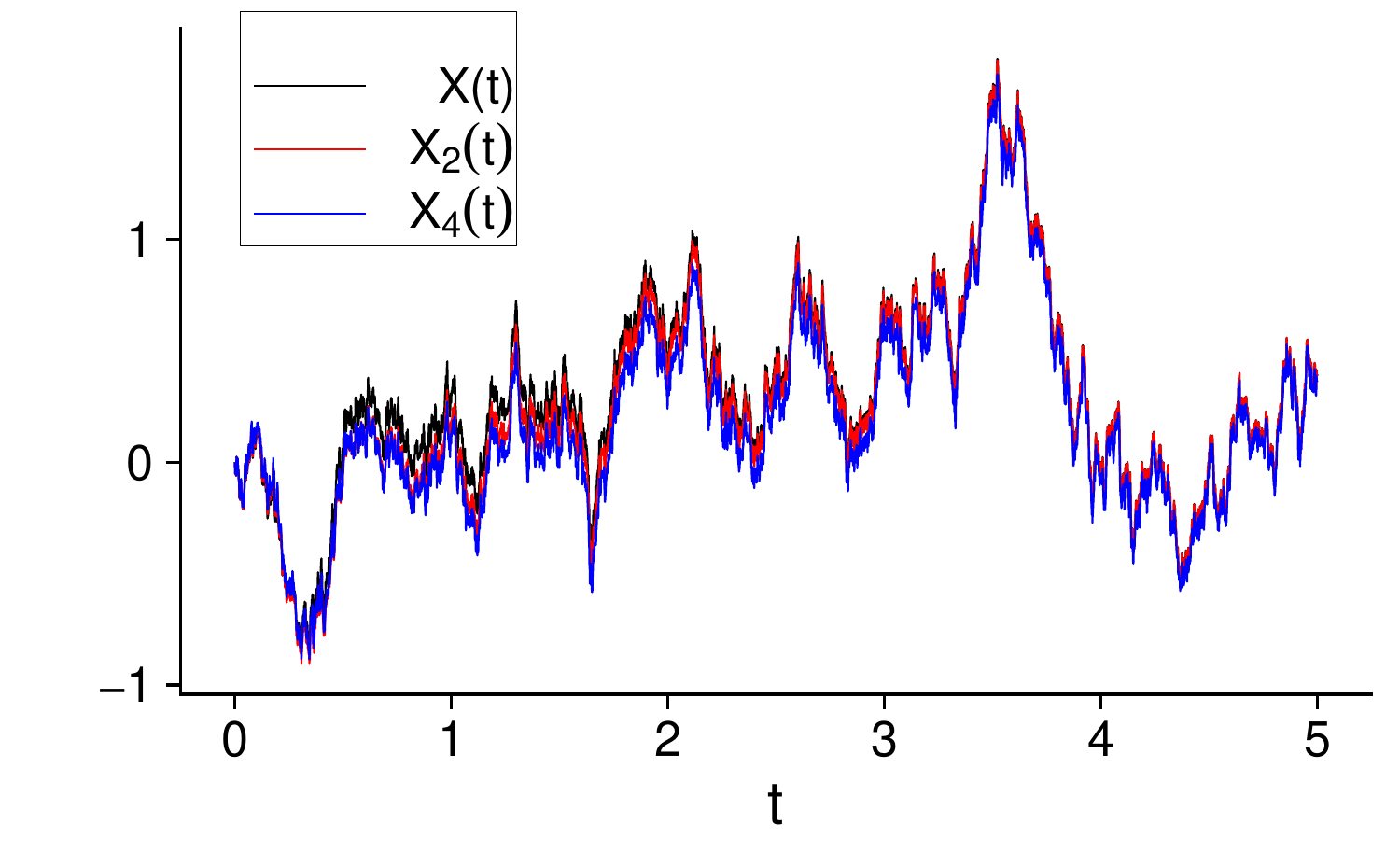}
	\includegraphics[width=0.49\linewidth]{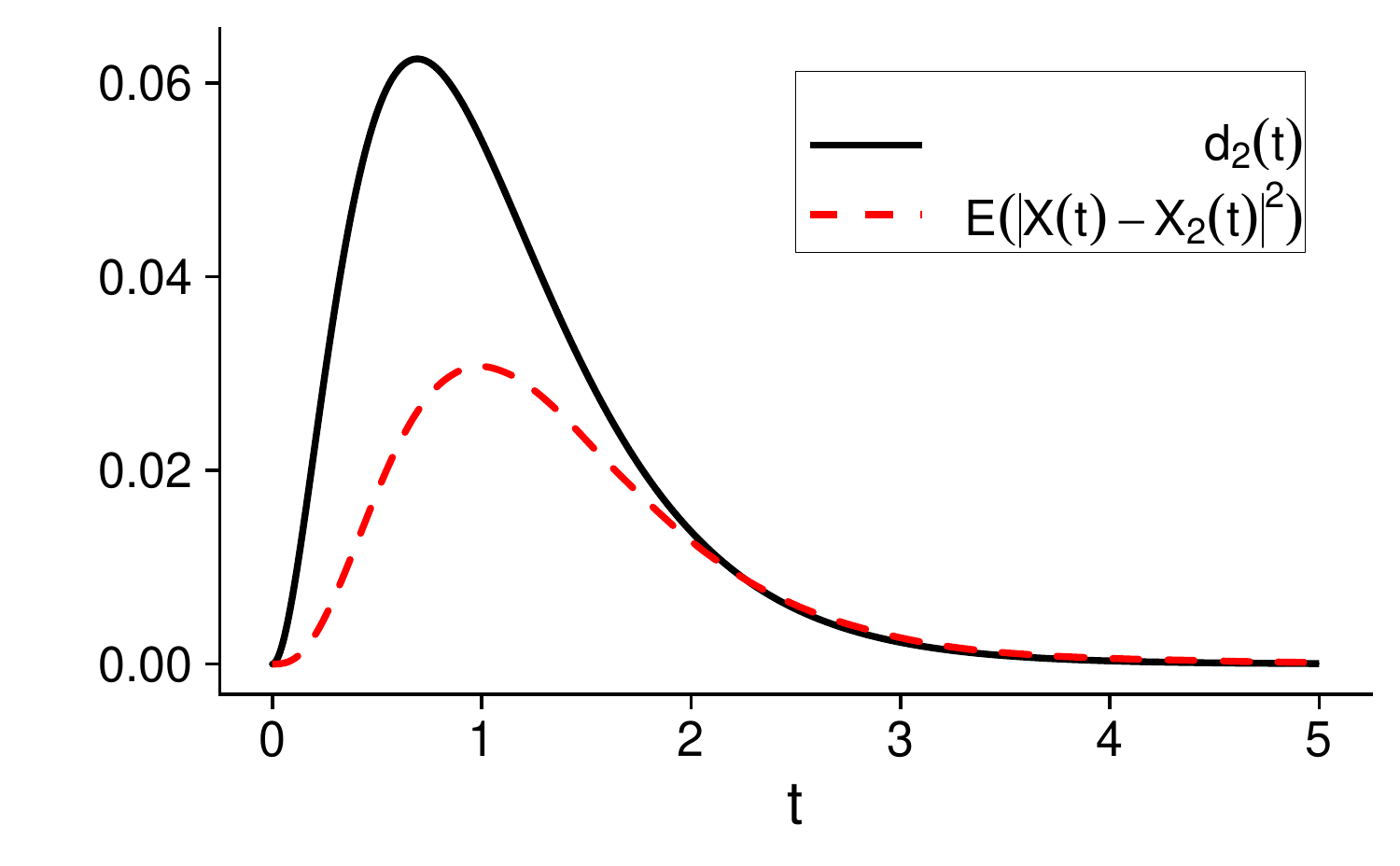}
	\caption{Left: sample paths of $X(t)$, $X_2(t)$ and $X_4(t)$ for the Single shot example. The trajectories are obtained by using Equation \eqref{spliteq} and simulating separately $Y(t)$ (which is common for the three processes), $Z(t)$ (from Equation \eqref{Zss}), $F_2(t)$ (from Equation \eqref{F2ss}) and $F_4(t)$ (solving numerically Equation \eqref{eulanpower4}). The parameters are chosen to be $\theta=1.5$, $\lambda=2$, $\sigma=1$ and $X_0=0$, with discretization time interval $\Delta t=10^{-3}$.\\
	Right: comparison between the simulated pointwise mean square error (of the approximation of $X(t)$ via $X_2(t)$) and its bound $d_2(t)$ from Equation \eqref{ssd2}. The parameters are chosen to be $\theta=1.5$, $\lambda=2$, $\sigma=1$ and $X_0=0$, with discretization time interval $\Delta t=10^{-3}$. To obtain $\E[(X(t)-X_2(t))^2]$, $N=10^4$ trajectories have been simulated.\\ 
	These simulations and the following ones are made by using the software environment \texttt{R} \cite{Rcode}.}
	\label{fig:f2f4}
\end{figure}
\subsection{A Poisson process as drift.}\label{Sec.4.2}
Let $z(t)= N(t)$ be the stochastic drift process with $N(t)$ a Poisson process with parameter $\lambda$ adapted to the filtration $\{\cF_t\}_{t \ge 0}$.
In this case one has $f_2(t)=\E[z(t)]=\mathbb{D}[z(t)]=\lambda t$. By definition of $Z(t)$ in Equation \eqref{Equ.2.3} we have
\begin{equation*}
Z(t)=e^{-\theta t}\int_0^tN(s)e^{\theta s}ds=\sum_{i=1}^{+\infty}\frac{1-e^{-\theta(t-\cT_i)}}{\theta}\chi_{[\cT_i,+\infty)}(t),
\end{equation*}
where last equality follows from $N(t)=\sum_{i=1}^{+\infty}\chi_{[\cT_i,+\infty)}(t)$ with $\cT_i$ the jump times of the process $N(t)$, hence, from Proposition \ref{prop2.1} we have that
\begin{equation*}
X(t)=e^{-\theta t}x_0+\sigma e^{-\theta t}\int_0^{t}e^{\theta s}dW_s+\sum_{i=1}^{+\infty}\frac{1-e^{-\theta(t-\cT_i)}}{\theta}\chi_{[\cT_i,+\infty)}(t).
\end{equation*}
The process $X_2(t)$ is obtained by solving \eqref{SDEapproxp} with $f_2(t)=\lambda t$, obtaining
\begin{equation*}
X_2(t)=e^{-\theta t}x_0+\sigma e^{-\theta t}\int_0^{t}e^{\theta s}dW_s+\frac{\lambda t}{\theta}-\frac{\lambda}{\theta^2}(1-e^{-\theta t}).
\end{equation*}
Concerning the upper bound for the punctual $L^2$ distance, we have, by Equation \eqref{d2def},
\begin{equation*}
d_2(t)=\frac{\lambda t}{2\theta}-\frac{\lambda}{4\theta^2}(1-e^{-2\theta t}).
\end{equation*}
This time, $d_2$ does not belong to $L^1(0,+\infty)$. However, it is not difficult to check that there exists a constant $C_2>0$ such that $\cJ_2[X_2]\le C_2T^2$ for $T$ large enough.
\subsection{A Compound Poisson process as drift.}\label{Sec.4.3}
Let $z(t)=\sum_{i=1}^{N(t)}J_i$ be the stochastic drift process where $N(t)$ is a Poisson process adapted to the filtration $\{\cF_t\}_{t \ge 0}$ with parameter $\lambda>0$ and $\{J_i\}_{i \in \mathbb{N}}$ is a sequence of i.i.d. random variables, distributed as a given variable $J \in L^2(\bP)$, which are also independent from $N(t)$. Let us also suppose that $J_i$ are measurable with respect to $\cF_t$ for any $t>0$ and for any $i \in \N$. By definition of $Z(t)$ in Equation \eqref{Equ.2.3} we have
\begin{equation*}
Z(t)=\sum_{i=1}^{+\infty}J_i\frac{1-e^{-\theta(t-\cT_i)}}{\theta}\chi_{[\cT_i,+\infty)}(t),
\end{equation*}
where $\cT_i$ are the jump times of the process $z(t)$, hence, by Proposition \ref{prop2.1}, we have
\begin{align*}
X(t)=e^{-\theta t}x_0+\sigma e^{-\theta t}\int_0^{t}e^{\theta s}dW_s+\sum_{i=1}^{+\infty}J_i\frac{1-e^{-\theta(t-\cT_i)}}{\theta}\chi_{[\cT_i,+\infty)}(t).
\end{align*}
The process $X_2(t)$ is then obtained by solving Equation \eqref{SDEapproxp} with $f_2(t)=\lambda t \E[J]$:
\begin{equation*}
X_2(t)=e^{-\theta t}x_0+\sigma e^{-\theta t}\int_0^{t}e^{\theta s}dW_s+\frac{\lambda t \E[J]}{\theta}-\frac{\lambda \E[J]}{\theta^2}(1-e^{-\theta t}).
\end{equation*}
Moreover, since $\D[z(t)]=\lambda t \E[J^2]$, we have, by Equation \eqref{d2def},
\begin{equation*}
d_2(t)=\frac{\lambda \E[J]t}{2\theta}-\frac{\lambda\E[J^2]}{4\theta^2}(1-e^{-2\theta t}).
\end{equation*}
As for the case of the Poisson process, also in this case there exists a constant $C_2>0$ such that $\cJ_2[X_2]\le C_2T^2$ for $T$ large enough.
\subsection{A Shot Noise as drift.}\label{Sec.4.4}
Let $M$ be a $L^2(\bP)$ random variable with positive integer values measurable with respect to $\cF_t$ for any $t>0$, $\{\beta_i\}_{i \in \N}$ i.i.d. $L^2(\bP)$ random variables independent of $M$, measurable with respect to $\cF_t$ for any $t>0$, and distributed as $\beta$ and $\{\cT_i\}_{i \in \N}$ i.i.d. almost surely positive absolutely continuous random variables that are Markov times with respect to $\{\cF_t\}_{t \ge 0}$, distributed as a fixed random variable $\cT$ and independent of the $\beta_i$ and $M$. Moreover, let us consider a function $R(t)$ (called response function) such that $R(t)=0$ for any $t<0$. Let us denote by $p_{\cT}(t)$ the probability density function of $\cT$. Let us consider the stochastic process $z(t)=\sum_{i=1}^{M}\beta_iR(t-\cT_i)$ as drift process. 
In this case $Z(t)$, by Equation \eqref{Equ.2.3}, is given by
\begin{equation}\label{ZSN}
Z(t)=e^{-\theta t}\sum_{i=1}^{M}\beta_i \int_0^{t}R(s-\cT_i)e^{\theta s}ds
\end{equation}
and then the process $X(t)$, by Equation \eqref{spliteq},
\begin{equation*}
X(t)=e^{-\theta t}x_0+\sigma e^{-\theta t}\int_0^{t}e^{\theta s}dW_s+Z(t).
\end{equation*}
Now let us recall that
\begin{equation*}
f_2(t)=\E[z(t)]=\E[M]\E[\beta]\varphi(t),
\end{equation*}
where
\begin{equation}\label{varphiformula}
\varphi(t)=\E[R(t-\cT)]=\int_0^{t}R(t-s)p_{\cT}(s)ds=(R \ast p_{\cT})(t).
\end{equation}
Solving Equation \eqref{SDEapproxp} we get
\begin{equation}\label{X2SN}
X_2(t)=e^{-\theta t}x_0+\sigma e^{-\theta t}\int_0^{t}e^{\theta s}dW_s+e^{-\theta t}\E[M]\E[\beta]\int_0^{t}\varphi(s)e^{\theta s}ds.
\end{equation}
Finally, since
\begin{equation*}
\D[z(t)]=\E[\beta]^2\varphi^2(t)(\D[M]-\E[M])+\E[M]\E[\beta^2]\Psi(t),
\end{equation*}
where
\begin{equation}\label{PsiSN}
\Psi(t)=\E[R^2(t-\cT)]=(R^2\ast p_{\cT})(t),
\end{equation}
we obtain, from Equation \eqref{d2def},
\begin{equation}\label{eqd2shot}
d_2(t)=e^{-2\theta t}\E[\beta]^2(\D[M]-\E[M])\int_{0}^t\varphi^2(s)e^{2\theta s}ds+\E[M]\E[\beta^2]e^{-2\theta t}\int_0^t\Psi(s)e^{2\theta s}ds.
\end{equation}
Let us observe that if $M$ is distributed as a Poisson random variable with parameter $\lambda$, then $\D[M]-\E[M]=0$ and we have
\begin{equation*}
d_2(t)=\lambda\E[\beta^2]e^{-2\theta t}\int_0^t\Psi(s)e^{2\theta s}ds.
\end{equation*}
An interesting case is given by $R(t)=e^{-\frac{t}{\tau}}1_{[0,+\infty)}(t)$. In the neuronal modeling context, a process $z(t)$ of this kind goes under the name of {\it shot noise}. It plays a key role in the description of neuronal networks dynamics as described in the next section.
\subsection{A Brownian motion as drift.}\label{Sec.4.5}
Let $z(t)=\widetilde{W}(t)+\lambda t$ where $\widetilde{W}(t)$ is a Brownian motion adapted to $\{\cF_t\}_{t \ge 0}$ and independent of $W(t)$, and $\lambda \ge 0$. By definition of $Z(t)$ in Equation \eqref{Equ.2.3} we have
\begin{equation*}
Z(t)=e^{-\theta t}\int_0^{t}(\widetilde{W}(s)+\lambda s)e^{\theta s}ds
\end{equation*}
that solves the equation
\begin{equation*}
dZ(t)=(-\theta Z(t)+\widetilde{W}(t)+\lambda t)dt, \qquad Z(0)=0.
\end{equation*}
By Proposition \ref{prop2.1} we have
\begin{equation*}
X(t)=e^{-\theta t}x_0+\sigma e^{-\theta t}\int_0^{t}e^{\theta s}dW_s+\frac{\lambda t}{\theta}-\frac{\lambda}{\theta^2}(1-e^{-\theta t})+e^{-\theta t}\int_0^{t}\widetilde{W}(s)e^{\theta s}ds.
\end{equation*}
Since $f_2(t)=\E[z(t)]=\lambda t$, solving \eqref{SDEapproxp}, we have
\begin{equation*}
X_2(t)=e^{-\theta t}x_0+\sigma e^{-\theta t}\int_0^{t}e^{\theta s}dW_s+\frac{\lambda t}{\theta}-\frac{\lambda}{\theta^2}(1-e^{-\theta t}).
\end{equation*}
Note that $X_2$ is the same as the approximant we obtain in the Poisson case. However, since $\D[z(t)]=t$, we have, by Equation \eqref{d2def},
\begin{equation*}
d_2(t)=\frac{t}{2\theta}-\frac{1-e^{-2\theta t}}{4\theta^2}
\end{equation*}
which is independent of $\lambda>0$. Thus if $\lambda>1$, the upper bound given by $d_2(t)$ is stricter than the one in the Poisson process case; vice-versa if $\lambda<1$.\\
Concerning $X_4$, in general, if hypothesis $iii$ of Proposition \ref{prop3.10} does not hold, one could check if Equation \eqref{eulanpower4} admits a triple zero only for an at most finite set of $t \in [0,T]$. However, the case of a Gaussian drift term is particular since we can show the following Proposition:
\begin{prop}\label{propGauss}
	If $z(t)$ is a Riemann-integrable Gaussian process and hypotheses $i$ and $ii$ of Proposition \ref{prop3.10} for $p=2n$ for some $n \in \N$ hold, then Problem \eqref{prob} with running cost $J_{2n}$ and constant final cost, or final cost $\Phi_{2n}$, admits a unique solution $F_{2n}=F_2$.
\end{prop}
\begin{proof}
	Let us first observe that if $z(t)$ satisfies hypotheses $i$ and $ii$ of Proposition \ref{prop3.10} for $p=2n$, then it satisfies the same hypotheses for $p=2$. Moreover, being $z(t)$ a Riemann-integrable Gaussian process, also $Z(t)$ is a Gaussian process. Now let us observe that Equation \eqref{eulanpower} for $p=2n$ becomes
	\begin{equation}\label{eulanpowereven}
	\E[(Z(t)-F_{2n}(t))^{2n-1}]=0.
	\end{equation}
	Let us add and subtract $\E[Z(t)]$ in the left-hand side of Equation \eqref{eulanpowereven} to achieve
	\begin{align}\label{conti1}
	\begin{split}
	&\E[(Z(t)-F_{2n}(t))^{2n-1}]=\E[(Z(t)-\E[Z(t)]+\E[Z(t)]-F_{2n}(t))^{2n-1}]\\
	&\qquad=\sum_{i=0}^{2n-1}\binom{2n-1}{i}\E[(Z(t)-\E[Z(t)])^{2n-1-i}](\E[Z(t)]-F_{2n}(t))^{i}\\
	&\qquad=\sum_{i=1}^{2n-1}\binom{2n-1}{i}\E[(Z(t)-\E[Z(t)])^{2n-1-i}](\E[Z(t)]-F_{2n}(t))^{i}\\
	&\qquad \qquad+\E[(Z(t)-\E[Z(t)])^{2n-1}].
	\end{split}
	\end{align}
	However, being $Z(t)$ Gaussian, we have $\E[(Z(t)-\E[Z(t)])^{2n-1}]=0$ and then Equation \eqref{conti1} becomes
	\begin{align}\label{conti2}
	\begin{split}
	&\E[(Z(t)-F_{2n}(t))^{2n-1}]=\sum_{i=1}^{2n-1}\binom{2n-1}{i}\E[(Z(t)-\E[Z(t)])^{2n-1-i}](\E[Z(t)]-F_{2n}(t))^{i}\\
	&\qquad=\sum_{i=0}^{2n-2}\binom{2n-1}{i+1}\E[(Z(t)-\E[Z(t)])^{2n-2-i}](\E[Z(t)]-F_{2n}(t))^{i+1}\\
	&\qquad=(\E[Z(t)]-F_{2n}(t))\sum_{i=0}^{2n-2}\binom{2n-1}{i+1}\E[(Z(t)-\E[Z(t)])^{2n-2-i}](\E[Z(t)]-F_{2n}(t))^{i}.
	\end{split}
	\end{align}
	Thus, substituting the result of Equation \eqref{conti2} in Equation \eqref{eulanpowereven} we get
	\begin{equation*}
	(\E[Z(t)]-F_{2n}(t))\sum_{i=0}^{2n-2}\binom{2n-1}{i+1}\E[(Z(t)-\E[Z(t)])^{2n-2-i}](\E[Z(t)]-F_{2n}(t))^{i}=0
	\end{equation*}
	whose solution is given by $F_{2n}(t)=\E[Z(t)]=F_2(t)$, concluding the proof.
\end{proof}
Hence the processes $X_2$ and $X_4$ (and any $X_{2n}$ for $n \in \N$) coincide in the case of a Gaussian drift.\\
Finally, let us observe that even in this case $\cJ_2[X_2]$ can grow at most quadratically with respect to the time horizon $T$.
\subsection{An Ornstein-Uhlenbeck process as drift.}\label{Sec.4.6}
Let $z(t)=U(t)$ be the stochastic drift process with $U(t)$ the OU process solution of the following SDE
\begin{equation*}\label{Equ.4.1}
dU(t)=-\lambda U(t)dt+\sigma_U d\widetilde{W}(t),\quad U(0)=U_0
\end{equation*}
where $U_0$, $\sigma_U$, $\lambda \in \mathbb R$ with $\lambda \not = \theta$.\\
In this case we have, by Equation \eqref{Equ.2.3},
\begin{equation*}
Z(t)=e^{-\theta t}\int_0^t U(s) e^{\theta s}ds,
\end{equation*}
that is solution of
\begin{equation*}
dZ(t)=(-\theta Z(t)+U(s))ds,
\end{equation*}
and, by Prop \ref{prop2.1} and Equation \eqref{spliteq},
\begin{align*}
X(t)=e^{-\theta t}x_0+\sigma e^{-\theta t}\int_0^{t}e^{\theta s}dW_s+Z(t).
\end{align*}
Recalling that $f_2(t)=\mathbb{E}[z(t)]=e^{-\lambda t}U_0$ we have $X_2(t)$, by solving Equation \eqref{SDEapproxp}, as
\begin{equation*}
X_2(t)=e^{-\theta t}x_0+\sigma e^{-\theta t}\int_0^{t}e^{\theta s}dW_s+\frac{U_0}{\theta-\lambda}(e^{-\lambda t}-e^{-\theta t}).
\end{equation*}
Since we also have
\begin{equation*}
\mathbb{D}[z(t)]=\frac{ \sigma_U^2}{2\lambda}\left(1-e^{-2\lambda t}\right)
\end{equation*}
we obtain, by Equation \eqref{d2def},
\begin{equation*}
d_2(t)=\frac{\sigma_U^2}{4\lambda\theta}(1-e^{-2\theta t})+\frac{\sigma_U^2}{4\lambda(\theta-\lambda)}(e^{-2\lambda t}-e^{-2\theta t}).
\end{equation*}
Let us observe that in such case the mean cumulative error of approximation is asymptotically bounded by a constant, i.e.
\begin{equation*}
\frac{\cJ_2[X_2(t)]}{T}\le \frac{\sigma_U^2}{4\lambda \theta}+\frac{\sigma^2(\theta-\lambda)}{8\lambda^2\theta^2 T}+\frac{\sigma_U^2(2\theta-\lambda)}{8\lambda \theta^2(\theta-\lambda)}\frac{e^{-2\theta T}}{T}-\frac{\sigma_U^2}{8 \lambda^2(\theta-\lambda)}\frac{e^{-2\lambda T}}{T}=:D_2(T),
\end{equation*}
where $\lim_{T \to +\infty}D_2(T)=\frac{\sigma_U^2}{4 \lambda \theta}$, thus $\cJ_2[X_2]$ can grow at most linearly with respect to the time horizon.
%

\subsection{Numerical results}\label{SecNR}
\begin{table}[h]
	\begin{tabular}{|l|l|l|l|l|}
		\hline
		& $\mathcal{J}_2[X_2]$ & $\mathcal{J}_2[X_4]$ & $\mathcal{J}_4[X_2]$ & $\mathcal{J}_4[X_4]$ \\ \hline
		Single Shot                & $0.04809342$         & $0.07091437$         & $0.005860227$        & $0.003705015$        \\ \hline
		Poisson                    & $7.268493$           & $7.437058$           & $50.80439$           & $49.55527$           \\ \hline
		Compound Poisson           & $3.612515$           & $3.957901$           & $16.41228$           & $14.52411$           \\ \hline
		\textbf{Brownian Motion} & $3.619208$           & $3.619208$           & $12.40195$           & $12.40195$           \\ \hline
		\textbf{Ornstein-Uhlenbeck}         & $0.1985489$          & $0.1985489$          & $0.02649128$         & $0.02649128$         \\ \hline
	\end{tabular}
	\caption{Numerical evaluation of $\cJ_i[X_j]$ for $i=2,4$ and $j=2,4$ for the considered examples.}
	\label{table}
\end{table}
We considered six examples of possible choices for the drift. In the first four cases the resulting process $X$ is not  Gaussian, while in the last two it is. For each example (except the Shot Noise case), we evaluate numerically the quantities $\cJ_i[X_j]$, for $i=2,4$ and $j=2,4$. $\cJ_2$ is the integral mean square error of approximation of $X$ with the process $X_j$, while $\cJ_4$ is the integral mean fourth-power error of approximation of $X$ with the process $X_j$. In general we expect, from Theorem \ref{thmmain}, $\cJ_i[X_j] \ge \cJ_i[X_i]$ for $i=2,4$, $j=2,4$ and $i \not = j$. As shown in Table \ref{table}, this inequality holds strictly in the first three cases, while it is an equality on the last two (in bold in Table \ref{table}). This latter fact is justified by Proposition \ref{propGauss}, since $X_2=X_4$ for Gaussian drifts. To obtain each quantity, we used the formula
\begin{equation*}
\cJ_i[X_j]=\int_0^{T}\E[|Z(t)-F_j(t)|^i]dt
\end{equation*}
to avoid the simulation of the whole trajectories of the processes $X(t)$, $X_2(t)$ and $X_4(t)$, which could lead to numerical errors. While $F_2(t)$ is known explicitly, $F_4(t)$ has been obtained by solving numerically Equation \eqref{eulanpower4}. Since \eqref{eulanpower4} is a simple polynomial equation, we used for each $t$ the bisection method to evaluate $F_4(t)$, with a precision of $10^{-15}$. We set $\theta=1.5$, $\lambda=2$, $X_0=0$, $U_0=1$, $\sigma=1$, $\sigma_U=1$. In the compound Poisson case, we used $J \sim Exp(\lambda_J)$ with $\lambda_J=2$. The simulation time step is $\Delta t=10^{-3}$ and the time horizon $T=5$. For each example $10^4$ sample paths have been produced.\\
The case of the Shot Noise will be discussed in the next section.
\section{A model of a neuron embedded in a neuronal network}\label{Sec.5} 
In this section we will focus on an application of Theorem \ref{thmmain} to neuronal modeling described by the linear Equation \eqref{Equ.2.1}.\\
In particular, we are interested in the dynamics of a neuron embedded in a network of $M \in \N$ neurons. We assume that the neuron under study receives impulses from the other neurons, whenever they fire for the first time. We say that a neuron fires when its membrane potential exceeds a critical value: after the crossing, the value of the membrane potential is reset to its resting state and the dynamics starts anew. This process generates an electrical impulse that is transferred to the neurons that are connected to it. For this reason the membrane potential can be modelled as a leaky RC circuit with a drift characterizing the input stimuli. The membrane potential of each neuron of the network is then modelled by the following stochastic differential equation:
\begin{equation}
dV^{(i)}=\left(-\theta_i V^{(i)}+\mu_i\right)dt+\sigma_i dW_i \quad V^{(i)}(0)=v_{0,i},  \quad (i=1,\ldots, M),\label{Equ.5.1}
\end{equation} 
where $1/\theta_i>0$ is the characteristic time of the membrane, $\mu_i$ is a constant injected stimulus and $\sigma_i$ determines the amplitude of the baseline noise.\\

\begin{figure}
	\begin{center}
	\resizebox{0.7\textwidth}{!}{\usetikzlibrary{arrows.meta}
\begin{tikzpicture}
\draw[fill=yellow] (5,1) circle (1);
\draw[fill=green] (1,4) circle (0.5);
\draw[fill=green] (3,4) circle (0.5);
\draw[fill=green] (5,4) circle (0.5);
\draw[fill=green] (9,4) circle (0.5);
\draw[fill=black] (5,2) circle (0.05);
\draw[fill=black] (1,3.5) circle (0.05);
\draw[fill=black] (3,3.5) circle (0.05);
\draw[fill=black] (5,3.5) circle (0.05);
\draw[fill=black] (9,3.5) circle (0.05);
\draw[-{Latex[length=0.2cm]}] (1,3.5) -- (5,2);
\draw[-{Latex[length=0.2cm]}] (3,3.5) -- (5,2);
\draw[-{Latex[length=0.2cm]}] (5,3.5) -- (5,2);
\draw[-{Latex[length=0.2cm]}] (9,3.5) -- (5,2);
\draw (1,4) node {\Large $V_1$};
\draw (3,4) node {\Large $V_2$};
\draw (5,4) node {\Large $V_3$};
\draw (7,4) node {\Large $\cdots$};
\draw (9,4) node {\Large $V_M$};
\draw (5,1) node {\Large $V$};
\end{tikzpicture}}
\end{center}
	\caption{Schematization of the neuronal model in Section \ref{Sec.5}}
	\label{fig3}
\end{figure}
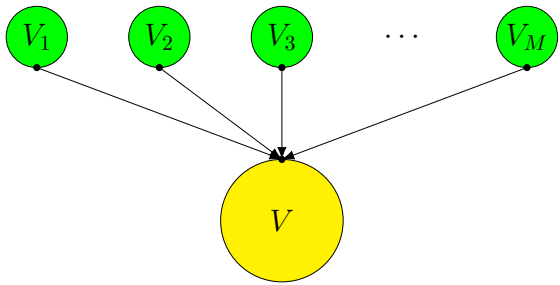
Concerning the embedded neuron, we assume that the stimuli it receives can be described by a function that is exponentially decreasing in time, with a characteristic time $1/\lambda>0$ (suppose for simplicity that $\lambda \not = \theta$). The initial amplitude of each stimulus is stochastic, represented by a family of i.i.d. random variables $\{\beta_i\}_{i \le M}$. The stochastic differential equation that describes the dynamics of this neuron is of shot noise type:
\begin{equation}
dV=\left(-\theta V+\sum_{i=1}^{M}\beta_{i} e^{-\lambda(t-\cT_{i})}\chi_{[0,+\infty)}(t-\cT_{i}) \right)dt+\sigma dW, \quad V(0)=v_0, \label{Equ.5.2}
\end{equation}
where $\cT_{i}$ is the first firing time of the $i$-th neuron of the network. In Figure \ref{fig3} we have a schematization of this model. Equations \eqref{Equ.5.1} and \eqref{Equ.5.2} are the classical It\^o stochastic differential equation for the stochastic diffusion Leaky Integrate and Fire (LIF) model (see, for instance, \cite{abbott,tuckwellvol2,sacerdote}). In particular, Equation \eqref{Equ.5.2} admits a stochastic drift of the form
\begin{equation}\label{Equ.5.3}
z(t):=\sum_{j=1}^{M}\beta_{i} e^{-\lambda(t-\cT_{i})}\chi_{[0,+\infty)}(t-\cT_{i}),
\end{equation}
where we suppose that $\{\beta_i\}_{i \le M}$ and $\{\cT_i\}_{i \le M}$ are i.i.d. and independent of each other. This is the driving term of the stimulus that neuron under study $V$ receives from the first-layer neurons $V_i$ (see Figure \ref{fig3}). All the other smaller inputs and changes in the environment are summarized by the Brownian noise. With the independence assumption we are supposing that the $M$ neurons described by the processes $\{V_i\}_{i \le M}$ are not communicating. This is a reasonable assumption: such physiological behaviour is common in the synaptic organization of the sensory neurons. An example of such behaviour is given by the sensory neurons of the olfactory bulb: such neurons are homogeneous and, neglecting eventual ephaptic coupling (which in general is insufficient to stimulate an action potential), independent from each other until their axons form a spherical structure named glomerulus, which carries all of such stimulus and is connected to the mitral cell \cite{ascione,Mori, Shepherd}. Another example of this behaviour is given by the photoreceptors of the retina, which are independent from each other and only linked to the retinal horizontal cell \cite{Shepherd,keener}.\\
$z(t)$ in Equation \eqref{Equ.5.3} is an example of shot noise with response function given by
\begin{equation}\label{Responseexp}
R(t)=e^{-\lambda t}\chi_{[0,+\infty)}(t).
\end{equation}
Let us then observe that, by Equation \eqref{ZSN}, with the choice of $R$ given by \eqref{Responseexp}, we get
\begin{equation*}
Z(t)=\frac{1}{\theta-\lambda}\sum_{i=1}^{M}\beta_i(e^{-\lambda(t-\cT_i)}-e^{-\theta(t-\cT_i)})\chi_{[0,+\infty)}(t-\cT_i)
\end{equation*}
and then, by solving Equation \eqref{Equ.5.2}, we obtain
\begin{align*}
V(t)&=e^{-\theta t}v_0+\sigma e^{-\theta t}\int_0^t e^{\theta s}dW_s\\
&\qquad+\frac{1}{\theta-\lambda}\sum_{i=1}^{M}\beta_i(e^{-\lambda(t-\cT_i)}-e^{-\theta(t-\cT_i)})\chi_{[0,+\infty)}(t-\cT_i).
\end{align*}
Now we consider some choices for the distribution of the first firing times $\cT_i$, as proposed in the literature, and we show the corresponding approximating processes $V_2$.
\subsection{The Exponential Case.} \label{Sec.5.1}
Let us suppose that $\cT$ is an exponential random variable (see, e.g., \cite{MCAP2011}) with parameter $\nu$. By definition of $\varphi(t)$ in Equation \eqref{varphiformula}, of $\Psi(t)$ in Equation \eqref{PsiSN}, and the choice of the response function $R(t)$ as in Equation \eqref{Responseexp}, we get
\begin{equation}\label{phiPsiexp}
\varphi(t)=\frac{\nu}{\nu-\lambda}(e^{-\lambda t}-e^{-\nu t}), \qquad \Psi(t)=\frac{\nu}{\nu-2\lambda}(e^{-2\lambda t}-e^{-\nu t}) 
\end{equation}
and then, by using Equation \eqref{X2SN}, we get
\begin{equation*}
V_2(t)=e^{-\theta t}v_0+\sigma e^{-\theta t}\int_0^t e^{\theta s}dW_s+\frac{M\E[\beta]\nu}{\nu-\lambda}\left(\frac{e^{-\lambda t}-e^{-\theta t}}{\theta-\lambda}-\frac{e^{-\nu t}-e^{-\theta t}}{\theta-\nu}\right).
\end{equation*}
Moreover, by using Equation \eqref{eqd2shot}, we get
\begin{align*}
d_2(t)&=M\nu\left(\frac{\E[\beta^2]}{\nu-2\lambda}\left(\frac{e^{-2\lambda t}-e^{-2\theta t}}{2(\theta-\lambda)}-\frac{e^{-\nu t}-e^{-2\theta t}}{2(\theta-\lambda)}\right)\right.\\
&\left.-\frac{\E[\beta]^2\nu}{(\nu-\lambda)^2}\left(\frac{e^{-2\lambda t}-e^{-2\theta t}}{2(\theta-\lambda)}-2\frac{e^{-(\lambda+\nu) t}-e^{-2\theta t}}{2\theta-\lambda-\nu}+\frac{e^{-2\nu t}-e^{-2\theta t}}{2(\theta-\nu)}\right)\right),
\end{align*}
with $\lim_{t \to +\infty}d_2(t)=0$.
Moreover, $d_2 \in L^1(0,+\infty)$, thus the approximation error $\cJ_2[V_2]$ is bounded by $\Norm{d_2}{L^1(0,+\infty)}$.
\subsection{The Gamma Case.} \label{Sec.5.2}
Gamma distribution is also a popular choice for the interspike interval distribution (see, e.g., \cite{Lansky2016}). Let us consider $\cT$ as a Gamma random variable with rate $\nu>0$ and shape parameter $\alpha>0$, i.e $\cT \sim \Gamma(\nu,\alpha)$. By definition of $\varphi(t)$ in Equation \eqref{varphiformula}, of $\Psi(t)$ in Equation \eqref{PsiSN}, and the choice of the response function $R(t)$ as in Equation \eqref{Responseexp}, we get
	\begin{equation}\label{phiPsiGamma}
	\varphi(t)=\left(\frac{\nu}{\nu-\lambda}\right)^{\alpha}\frac{e^{-\lambda t}\gamma(\alpha,(\nu-\lambda)t)}{\Gamma(\alpha)}, \quad \Psi(t)=\left(\frac{\nu}{\nu-2\lambda}\right)^{\alpha}\frac{e^{-2\lambda t}\gamma(\alpha,(\nu-2\lambda)t)}{\Gamma(\alpha)},
	\end{equation}
	where $\gamma(\alpha,t)$ is the lower incomplete Gamma function
	\begin{equation*}
	\gamma(\alpha,t)=\int_{0}^{t}s^{\alpha-1}e^{-s}ds.
	\end{equation*}
	By using Equation \eqref{X2SN}, we get
	\begin{align*}
	V_2(t)&=e^{-\theta t}v_0+\sigma e^{-\theta t}\int_0^t e^{\theta s}dW_s\\&+\frac{M\E[\beta]\nu^\alpha}{\lambda-\theta}\left(\frac{e^{-\theta t}}{(\nu-\theta)^\alpha}\gamma(\alpha,(\nu-\theta)t)-\frac{e^{-\lambda t}}{(\nu-\lambda)^\alpha}\gamma(\alpha,(\nu-\lambda)t)\right).
	\end{align*}
	$d_2(t)$ can be obtained from Equation \eqref{eqd2shot} by using $\varphi(t)$ and $\Psi(t)$ given in Equation \eqref{phiPsiGamma}. The expression is omitted here due to its length, but one can show that $\lim_{t \to +\infty}d_2(t)=0$ and $d_2 \in L^1(0,+\infty)$, thus $\cJ_2[V_2]$ is bounded by $\Norm{d_2}{L^1(0,+\infty)}$. 
	In Table \ref{table2} we show some numerical evaluations of $\cJ_i[V_j]$ for $i=2,4$ and $j=2,4$. As expected from Theorem \ref{thmmain}, we have $\cJ_i[V_j] > \cJ_i[V_i]$ for $i \not = j$.

\begin{table}[]
\begin{center}
	\begin{tabular}{|l|l|l|l|l|}
		\hline
		& $\mathcal{J}_2[V_2]$ & $\mathcal{J}_2[V_4]$ & $\mathcal{J}_4[V_2]$ & $\mathcal{J}_4[V_4]$ \\ \hline
		Exponential Case & $26.8471$            & $27.67105$           & $61.20081$           & $58.75134$           \\ \hline
		Gamma Case       & $26.4166$            & $27.65103$           & $52.46327$           & $49.15484$           \\ \hline
		No Assigned Distribution       & $4.387099$            & $4.396086$           & $4.311031$           & $4.315560$           \\ \hline
	\end{tabular}
\end{center}
	\caption{Numerical evaluation of $\cJ_i[V_j]$ for $i=2,4$ and $j=2,4$ for the considered examples of Subsections \ref{Sec.5.1} and \ref{Sec.5.2}. The parameters are chosen as: $\theta=1/10 \text{ ms}^{-1}$, $\lambda=1 \text{ ms}^{-1}$, $\nu=1/15 \text{ ms}^{-1}$, $\sigma=1 \text{mV ms}^{-1/2}$, $\alpha=2$, $v_0=0 \text{mV}$, $M=10$. $\beta_i$ are chosen to be uniformly distributed in $(0.5 \text{mV},1.5 \text{mV})$. Last example is obtained without assuming any assigned distribution on $\cT_i$, but simulating Equation \eqref{Equ.5.1} for $i=1,\dots,10$ with $\mu_i=6 \text{mV ms}^{-1}$, $\sigma_i=1 \text{mV ms}^{-1/2}$, $\theta_i=1/10 \text{ ms}^{-1}$, $v_{0,i}=0 \text{mV}$ and $\cT_i$ is the first passage time of $V_i$ through the threshold $V_{th}=20 \text{mV}$. The time horizon is fixed at $T=10 \text{ms}$ while the simulation time step is $\Delta t=10^{-2} \text{ms}$. For each evaluation $N=10^{4}$ sample paths have been produced and $F_4(t)$ is obtained for each $t$ by solving equation \eqref{eulanpower4} via bisection method with a precision of $10^{-15}$.}
	\label{table2}
\end{table}

\section{Concluding remarks}\label{Sec.6}
In this work we studied the problem of approximating solutions of linear SDEs with stochastic drift by using Ornstein-Uhlenbeck type processes, as introduced in Section \ref{Sec.2}. In particular, in Section \ref{Sec.3}, we showed sufficient and necessary conditions for existence and uniqueness of an optimal approximation (with respect to a suitable, but general, cost functional). Moreover, in Subsection \ref{Sec.3.7}, we show that a wide class of cost functionals (i.e. the power cost functionals with $p \ge 2$) satisfies Theorem \ref{thmmain}.\\
In Section \ref{Sec.4}, these results have been applied to some examples that are of interest in the classical literature. Some specific features of the approximations are highlighted in such examples. For instance, for the simplest example in Subsection \ref{Sec.4.1}, we plot the simulated sample paths for $X$, $X_2$ and $X_4$ in Figure \ref{fig:f2f4}, on the left. On the right of Figure \ref{fig:f2f4} we also plotted the pointwise mean square error $\E[|X(t)-X_2(t)|^2]$ together with its bound $d_2(t)$.
The examples in Subsections \ref{Sec.4.2}, \ref{Sec.4.3} and \ref{Sec.4.5} exhibit the same (eventually up to some constant) approximating process with respect to the quadratic cost, highlighting the exclusive dependence on the mean of the drift in such case. However, the performance of the approximation is strictly related to the variance of the processes describing the drift. In Subsection \ref{Sec.4.6} we also provided an upper bound for a temporal-mean of the mean-square error, showing in this case that such mean is bounded by a constant. It is actually easy to show that this behavior appears every time the function $d_2(t)$ is in $L^\infty(\R)$, equivalently if the drift process concentrates around its asymptotic mean. Moreover, in Subsection \ref{Sec.4.5} we proved that for Gaussian drift terms the approximations $X_2$ and $X_{2n}$ for any $n \in \N$ coincide. In Subsection \ref{SecNR} we compare the behaviour of $\cJ_2$ and $\cJ_4$ on $X_2$ and $X_4$ for the examples of Section \ref{Sec.4}, giving both a confirmation of Theorem \ref{thmmain} and some quantitative information on the approximation error.\\
Finally, the example in Subsection \ref{Sec.4.4} is of interest in the frame of neuronal modeling. Indeed, we provide a model for a single neuron embedded in a neuronal network in Section \ref{Sec.5}. In such case, we specialize the response function of the shot noise process and we study the approximation of the membrane potential process. Indeed, as we did for Section \ref{Sec.4}, we evaluated the approximation errors $\cJ_2$ and $\cJ_4$ on the processes $V_2$ and $V_4$ under some suitable assumptions on the spiking times of the first layer of neurons $V^{(i)}$ and finally without any assumptions on them, by simulating the whole network. \\
The novelty of our findings is that in our case no hypotheses of ergodicity of the process (\cite{liu},\cite{jacod},\cite{limnios}) is asked, nor we use slow-fast dynamic techniques (\cite{berglund}, \cite{roc}), nor we increase asymptotically the number of neurons as in the mean field theory approach (\cite{Delarue},\cite{Touboul}, \cite{faugeras2019meanfield}, \cite{faugeras2019}, \cite{zanco}).\\
If on one hand the lack of assumptions on the drift process makes the result very general, on the other hand our approach does not easily lead to explicit solutions, since Equation \eqref{eulan} could be impossible to solve in closed form and numerical evaluations are needed.\\
The study of the approximation problem with a stochastic time horizon, depending eventually on the process itself, as, for instance, first passage times through some fixed thresholds, will be the subject of future studies.
\section*{Acknowledgements}
\vspace*{-0.3cm}
We thank the anonymous reviewers whose comments have greatly improved this manuscript. This research is partially supported by MIUR - PRIN 2017, project Stochastic Models for Complex Systems, no. 2017JFFHSH, by Gruppo Nazionale per il Calcolo Scientifico (GNCS-INdAM), by Gruppo Nazionale per l'Analisi Matematica, la Probabilit\`a e le loro Applicazioni (GNAMPA-INdAM) and by the Czech Science Foundation project 20-10251S.

\section*{References}
\bibliographystyle{elsarticle-num}

\bibliography{mybibfile}

\end{document}